\documentclass{article}
\usepackage[utf8]{inputenc}
\usepackage{cite}
\usepackage{hyperref}
\usepackage[enableskew]{youngtab}

\title{The cohomology of rank two stable bundle moduli:\\ mod two nilpotency \& skew Schur polynomials}
\author{Christopher Scaduto \& Matthew Stoffregen}
\date{}

\usepackage[a4paper, total={6in, 8.5in}]{geometry}
\usepackage{booktabs}
\usepackage{graphicx}
\usepackage[dvipsnames]{xcolor}
\usepackage{amssymb}
\usepackage{mathtools}
\usepackage{tikz-cd}
\usepackage{titling}
\usepackage{amsthm}
\usepackage{amscd}
\usepackage{caption}
\usepackage{amsmath}
\usepackage{MnSymbol}
\usepackage{mathrsfs}  
\usepackage{tikz}
\usepackage{dsfont}
\usepackage{sectsty}

\sectionfont{\fontsize{12}{15}\selectfont}

\usetikzlibrary{matrix}

\usetikzlibrary{shapes,snakes}
\usetikzlibrary{knots}
\usetikzlibrary{matrix,arrows,decorations.pathmorphing}
\usetikzlibrary{positioning}

\usetikzlibrary{matrix,arrows,decorations.pathmorphing}
\usetikzlibrary{positioning}

\newcommand{\cn}{\mathsf{CN}}
\newcommand{\Z}{\mathbb{Z}}
\newcommand{\Q}{\mathbb{Q}}

\newtheorem{theorem}{Theorem}[section]
\newtheorem{prop}[theorem]{Proposition}
\newtheorem{lemma}[theorem]{Lemma}

\newtheorem{corollary}[theorem]{Corollary}

\begin{document}

\maketitle

\begin{abstract}
We compute cup product pairings in the integral cohomology ring of the moduli space of rank two stable bundles with odd determinant over a Riemann surface using methods of Zagier. The resulting formula is related to a generating function for certain skew Schur polynomials. As an application, we compute the nilpotency degree of a distinguished degree two generator in the mod two cohomology ring. We then give descriptions of the mod two cohomology rings in low genus, and describe the subrings invariant under the mapping class group action.
\end{abstract}

\vspace{.2cm}


\section{Introduction}

Let $\Sigma_g$ be a closed, oriented surface of genus $g$, and let $N_g$ be the moduli space of flat $SU(2)$ connections on $\Sigma_g$ having holonomy $-1$ around a single puncture. The space $N_g$ is smooth symplectic manifold of dimension $6g-6$, and twice the class of its symplectic form, denoted $\alpha$, is a generator of $H^2(N_g;\Z)$. If $\Sigma_g$ is given a complex structure, then $N_g$ may be identified with the moduli space of stable holomorphic bundles of rank 2 with fixed odd determinant.\\

The betti numbers of $N_g$ were first computed by Newstead \cite{newstead-top}, and Atiyah-Bott \cite{ab} later showed that $H^\ast(N_g;\Z)$ is torsion-free. Newstead also showed in \cite{newstead-char} that the cohomology ring is generated by integral classes $\alpha,\beta,\psi_1,\ldots,\psi_{2g}$ {\emph{over the rationals}}. Here $\beta$ is degree 4, and each $\psi_i$ is degree 3.
Newstead conjectured the relation $\beta^g=0$, which was proved by Thaddeus \cite{thaddeus-conformal} and Kirwan \cite{kirwan}. A beautiful presentation for the rational cohomology ring of $N_g$ was established by several \cite{baranovski, king-newstead, siebert-tian, zagier} following the work of Thaddeus \cite{thaddeus-conformal}.\\

The {\emph{nilpotency degree}} of an element $x$ in a ring is the smallest $n\geqslant 1$ such that $x^n=0$. In the integral cohomology ring, the nilpotency degree of $\beta$ is equal to $g$, while that of $\alpha$ is equal to $3g-4=\frac{1}{2}\dim N_g + 1$, since $\alpha$ is proportional to the symplectic form class. The situation is quite different with $\Z_2$-coefficients. First, the mod 2 reduction of $\alpha$ can be realized as $w_2(E)$ of an $SO(3)$-bundle $E$ over $N_g$ for which $\beta = p_1(E)$. By the general relation $w_2(E)^2 \equiv p_1(E)$ (mod 2),

\vspace{.1cm}
\[
    \alpha^2 \; \equiv \; \beta \mod 2.
\]
\vspace{.1cm}

\noindent In particular, $\beta$ is a redundant generator over $\Z_2$. Indeed, Atiyah-Bott \cite{ab} tell us that to generate the cohomology ring {\emph{over the integers}}, we need the classes $\alpha,\frac{1}{4}(\alpha^2-\beta),\psi_1,\ldots,\psi_{2g}$ and additional classes $\delta_1,\ldots,\delta_{2g-1}$. Here $\delta_i$ has degree $2i$. We will see that we only need the mod 2 reductions of $\alpha,\psi_1,\ldots,\psi_{2g}$ and $\delta_{2^i}$ for $2\leqslant 2^i\leqslant 2g-1$ in order to generate $H^\ast(N_g;\Z_2)$.\\

The moduli space $N_g$ embeds into the moduli space $M_g$ of projectively flat $U(2)$ connections on $\Sigma_g$ of fixed odd degree without fixed determinant. This is again a smooth symplectic manifold, now of dimension $8g-6$. It has a corresponding degree 2 class $a_1\in H^2(M_g;\Z)$ that restricts to $\alpha$. The nilpotency degrees of $\alpha$ and $a_1$ with $\Z_2$-coefficients are as follows.

\vspace{.6cm}

\begin{theorem}\label{thm:nilp} The nilpotency degree of $\alpha$ as viewed in $H^2(N_g;\Z_2)$ is equal to $g$:

\[  
\alpha^g\equiv 0 \text{{\emph{ (mod 2)}}}, \qquad \alpha^{g-1}\not\equiv 0 \text{{\emph{ (mod 2)}}}.
\]

\vspace{.25cm}

\noindent On the other hand, the nilpotency degree of $a_1$ as viewed in $H^2(M_g;\Z_2)$ is equal to $2g$:

\[  
a_1^{2g}\equiv 0 \text{{\emph{ (mod 2)}}}, \qquad a_1^{2g-1}\not\equiv 0 \text{{\emph{ (mod 2)}}}.
\]
\end{theorem}

\vspace{.6cm}

\noindent To establish that $\alpha^g$ is zero mod 2, we consider its cup product pairings with monomials in the generators listed above. The pairing formula will be expressed as the extraction of a coefficient from a formal power series whose coefficients are symmetric functions. To state the result, it is convenient to introduce the rational cohomology classes $\xi_i$ on $N_g$ which satisfy:

\vspace{.3cm}

\begin{equation}
    \xi_i \; = \; \sum_{j=0}^{i} {2g-1-j \choose i-j} \left(-\frac{\alpha}{2}\right)^{i-j}\delta_j, \qquad \qquad \delta_i \; = \; \sum_{j=0}^{i} {2g-1-j \choose i-j} \left(\frac{\alpha}{2}\right)^{i-j}\xi_j.\label{eq:defxi}
\end{equation}

\vspace{.3cm}

\noindent More precisely, the left-hand formula in (\ref{eq:defxi}) may be taken as the definition of $\xi_i$, and the right-hand formula is the induced inverse relation between the $\delta_i$ and $\xi_i$ generators. Note that $2^{i}\xi_i$ is an integral cohomology class for $N_g$ of degree $2i$. Next, we let $e_i$ denote the $i^\text{th}$ elementary symmetric function, and $m_\lambda$ the monomial symmetric function associated to a partition $\lambda$. Define $U(T) = \sum_{n\geqslant 0} m_{( 2^n 1)}(-T)^n$ as a power series with coefficients in the ring of symmetric functions. Here, the notation $(2^n 1)$ stands for the partition with $1$ one and $n$ two's. Also define

\vspace{.3cm}

\[
    Q(T) \; = \; e_1 + e_3 T + e_5 T^2 + e_7 T^3 + \cdots \; = \; \sum_{n\geqslant 0} e_{2n+1} T^n.
\]

\vspace{.3cm}

\noindent We write $x[N_g]$ for the evaluation of a top-degree integral cohomology class $x$ against the fundamental class of $N_g$. The following, along with (\ref{eq:defxi}), computes the pairings $\delta_{\lambda_1}\delta_{\lambda_2}\cdots\delta_{\lambda_k}[N_g]$, and is the main technical result of the paper.

\vspace{.5cm}

\begin{theorem}\label{thm:maincomp} Suppose $\lambda = (\lambda_1,\lambda_2,\ldots,\lambda_n)$ is a partition of $3g-3$. Then we have:

\vspace{.2cm}
\[
        \xi_{\lambda_1}\xi_{\lambda_2}\cdots \xi_{\lambda_n}[N_g]  \; = \;  \frac{1}{2^{g-1}}\cdot \underset{m_{\lambda}T^{g-1}}{\text{{\emph{Coeff}}}} \Big[\,U(T)^g/Q(T)\,\Big]
\]
\end{theorem}

\vspace{.5cm}

\noindent We obtain a similar formula for pairings on $M_g$. Since $\delta_1$ is a non-zero multiple of $\alpha$, Theorem \ref{thm:maincomp} can be used to compute pairings involving both powers of $\alpha$ and $\delta_i$ classes. In the sequel, we will also write down pairing formulas involving the $\psi_i$ classes. The proofs of these pairing formulas follow the computational framework of Zagier \cite{zagier}, whose starting point was Thaddeus's formula for the intersection pairings involving the Newstead classes $\alpha,\beta,\psi_1,\ldots,\psi_{2g}$.\\

As pointed out to the authors by Ira Gessel, up to some renormalizing, the reciprocal power series $1/Q(T)$ is a generating function for the skew Schur functions associated to a particular family of skew partitions. We briefly explain this. Let $\lambda(n,m)$ be the partition $(n,\cdots,n,n-1,n-2,\ldots,2,1)$ where $n$ appears $m$ times. Note that $\lambda(n,0)=(n-1,n-2,\ldots,2,1)$. In general, if $\lambda$ and $\mu$ are paritions, the skew partition $\lambda/\mu$ is pictorially the result of drawing the Young tableau for $\lambda$ and deleting the part of the tableau given by $\mu$. See Figure \ref{fig:skewtableau}. To any skew tableau $\lambda/\mu$ there is defined a skew Schur symmetric function $s_{\lambda/\mu}$. We will explain in Section \ref{sec:skew} the following identity:

\begin{figure}[t]
\centering
\includegraphics[scale=.85]{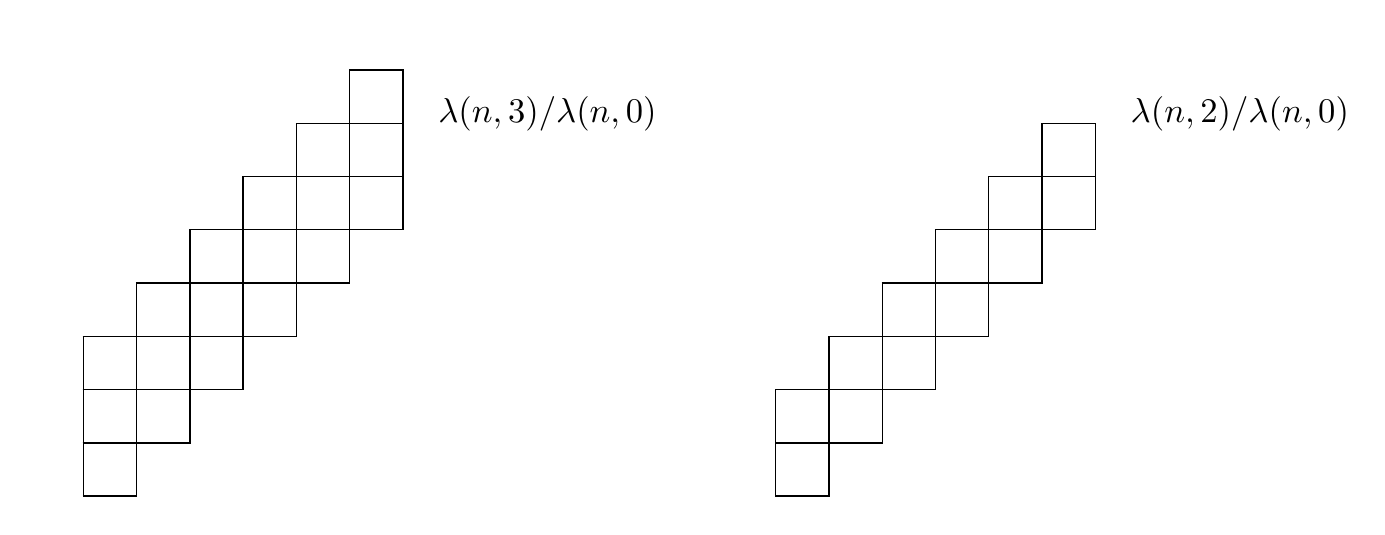}
\caption{{\small{With $n=6$, the skew tableaux appearing in the skew Schur functions and $1/Q(T)$ and $1/E(T)$.}}}\label{fig:skewtableau}
\end{figure}

\vspace{.3cm}

\begin{equation}
    1/Q(T) \; = \; \sum_{n \geqslant 0} e_1^{-n-1} s_{\lambda(n,3) / \lambda(n,0)} (-T)^n.\label{eq:qinv}
\end{equation}

\vspace{.3cm}

\noindent As the power series $U(T)$ is comparatively simple, we see that the complexity of the cup product pairings among the $\delta_i$ classes comes from the skew Schur functions $s_{\lambda(n,3) / \lambda(n,0)}$.\\

Theorem \ref{thm:maincomp} allows us to explicitly describe the ring $H^\ast(N_g;\Z_2)$ for low genus, and we do this in Section \ref{sec:comps}. Ideally, we would like to find presentations for these rings that are as nice as the recursive presentations for $H^\ast(N_g;\Q)$ as found by \cite{baranovski, king-newstead, siebert-tian, zagier}. The situation for non-rational coefficients seems more complicated, however, as our computations suggest.\\

The manifold $N_g$ may be given complex structure, and is in fact an example of a smooth Fano variety. In particular, in place of the $\delta_i$ above, we may consider the products of its Chern classes. Here another power series $E(T)$ with coefficients symmetric functions appears:

\vspace{.3cm}

\begin{equation*}
    E(T) \; = \; 1 + e_2T + e_4T^2 +e_6T^3+\cdots = \sum_{n\geqslant 0}e_{2n}T^n.
\end{equation*}

\vspace{.3cm}

\noindent The analogue of (\ref{eq:qinv}) is the relation $1/E(T)=\sum_{n\geqslant 0}s_{\lambda(n,2)/\lambda(n,0)}(-T)^n$. We then have the following, whose proof is very similar to the proof of Theorem \ref{thm:maincomp}:

\vspace{.5cm}

\begin{theorem}\label{thm:mainchern} Suppose $\lambda = (\lambda_1,\lambda_2,\ldots,\lambda_n)$ is a partition of $3g-3$. Set $c_i=c_i(TN_g)$. Then:

\vspace{.2cm}

\[
        c_{\lambda_1}c_{\lambda_2}\cdots c_{\lambda_n}[N_g]  \; = \;  (-2)^{3g-3}\cdot \underset{m_{\lambda}T^{g-1}}{\text{{\emph{Coeff}}}} \Big[\,U(T)^g/Q(T)E(T)\,\Big]
\]
\end{theorem}

\vspace{.5cm}

\noindent This theorem has the following application. From \cite[\S 4]{newstead-char}, we know that the total Pontryagin class of $N_g$ is equal to $(1+\beta)^{2g-2}$. The relation $\beta^g=0$ mentioned above then implies that all Pontryagin numbers of $N_g$ vanish. It is easy to see from Theorem \ref{thm:mainchern} that the Chern numbers of $N_g$ are all even, and hence all Stiefel-Whitney numbers of $N_g$ vanish. A theorem of Wall \cite{wall} says that two closed, oriented manifolds are oriented-cobordant if and only if they have the same Pontryagin and Stiefel-Whitney numbers. We then deduce the following, which we suspect was already known, but for which we could not find a reference:

\vspace{.5cm}

\begin{corollary}\label{cor:nullcob} The manifold $N_g$ is oriented null-cobordant.
\end{corollary}

\vspace{.5cm}

We make a few historical remarks. The classes $\delta_i$ are Chern classes of the direct image of a universal rank two complex bundle over the moduli space $N_g$. The Riemann-Roch formula gives expressions for its Chern classes in terms of the more basic classes $\alpha,\beta,\psi_j$. The direct image bundle has rank $2g-1$, and so the expressions one obtains for Chern classes in degrees higher than $2g-1$ are relations in the cohomology ring. Mumford is usually credited with conjecturing that these expressions, at least in the case of $M_g$, form a complete set of relations, see \cite[p. 582]{ab}. This was proved by Kirwan \cite{kirwan}. The beautiful recursive presentation for the rational cohomology ring found later by \cite{baranovski, king-newstead, siebert-tian, zagier} uses relations that are most naturally viewed as Chern classes of a bundle over $N_g$ induced by an embedding into a Grassmannian, see e.g. \cite[\S 1]{siebert-tian}. However, Zagier shows at the end of Section 6 in \cite{zagier} that they can also be recovered from the Chern classes of the direct image bundle.\\

The work presented here is motivated by a problem in instanton homology with mod two coefficients, and in particular, the analogue of Mu\~{n}oz's work \cite{munoz-ring} in characteristic two. The mod 2 instanton homology of a surface times a circle with non-trivial $SO(3)$-bundle should be a deformation of the ring $H^\ast(N_g;\Z/2)$, and should agree with a version of the quantum cohomology of the symplectic manifold $N_g$ with mod 2 coefficients. We expect the nilpotency degree of $\alpha$ as viewed in this deformation, perhaps in the ring modulo the $\psi_i$ classes, to be related to homology cobordism invariants defined in unpublished work by Fr\o yshov using mod 2 instanton homology. The analogue in rational coefficients is the nilpotency degree $\lceil g/2 \rceil$ of $\beta$ mod $\gamma$ that appears in F\o yshov's inequality \cite[Thm. 1]{froyshov} for his $h$-invariant. See also the related paper \cite{sc}. The authors plan to return to these motivations in forthcoming work. \\

In a spin-off article, we will use the computations here to study the mod two betti numbers of the {\emph{framed}} moduli space, which is an $SO(3)$ bundle over $N_g$. These betti numbers are determined by the ranks of the maps on $H^i(N_g;\Z_2)$ given by cup-product with the degree two class $\alpha$.\\

\vspace{0.7cm}

\noindent \textbf{Outline.} Background is provided in Section \ref{sec:backg}, as well as several useful results regarding generators for the cohomology rings of $N_g$ and $M_g$ for different coefficient rings. In Section \ref{sec:intersection} we review Thaddeus's pairing formula for the Newstead generators. In Section \ref{sec:maincomp} we prove Theorems \ref{thm:maincomp} and \ref{thm:mainchern} and discuss some of their implications, as well as the relationship with skew Schur functions. In Section \ref{sec:nilp} we prove Theorem \ref{thm:nilp}. Finally, in Section \ref{sec:comps} we present computations obtained using Theorem \ref{thm:maincomp}, and describe the mod two cohomology rings of $N_g$ for low genus.

\vspace{.7cm}

\noindent \textbf{Acknowledgments.} The authors thank Ira Gessel for a very helpful email correspondence. The first author was supported by NSF grant DMS-1503100.

\vspace{.6cm}

\newpage

\section{Background \& Generators}\label{sec:backg} In this section we describe sets of generators for the cohomology rings of $N_g$ and $M_g$ for different coefficient rings. We also write down generators for the subring of the cohomology of $N_g$ invariant under the mapping class group action. As we proceed, we will introduce some necessary background, but see \cite{thaddeus-intro} for a more proper introduction. We take a moment to emphasize here an important point about our notation regarding the generators $\delta_i$ and $d_i$ introduced below. Singling out a handle of the surface $\Sigma_g$ induces embeddings of $N_{g-1}$ and $M_{g-1}$ into $N_g$ and $M_{g}$, respectively.\\

\vspace{.3cm}

\noindent \textbf{Caution:} {\emph{The restriction of $\delta_i \in H^{2i}(N_g;\Z)$ is not equal to $\delta_i \in H^{2i}(N_{g-1};\Z)$. Similarly, the restriction of $d_i \in H^{2i}(M_g;\Z)$ is not equal to $d_i \in H^{2i}(M_{g-1};\Z)$.}}\\

\vspace{.3cm}

\noindent For this reason, in the sequel we sometimes write $\delta_{g,i}$ and $d_{g,i}$ for $\delta_i$ and $d_i$, respectively. Finally, we mention that the contents of this section are derived mostly from Atiyah-Bott \cite{ab}, with the help of some additional observations.

\vspace{.65cm}

\subsection{Integral generators for the cohomology of $M_g$}\label{sec:mg}

We begin by defining the Atiyah-Bott generators for the ring $H^\ast(M_g;\Z)$. Central to the discussion is a universal rank two holomorphic bundle $U_g \longrightarrow M_g \times \Sigma_g$. There is an ambiguity in the choice of this bundle: tensoring by any holomorphic line bundle over $M_g\times \Sigma_g$ produces another, possibly non-isomorphic, universal bundle. Atiyah-Bott fix their choice of universal bundle by starting with any universal $U_g$ and defining the following normalized bundle:

\vspace{.18cm}
\begin{equation}
    V_g \; := \; U_g\otimes f^\ast\left(\det\left( U_g|_{M_g}\right)^{\otimes g}\otimes \det\left(f_!U_g\right)\right).
\end{equation}
\vspace{.1cm}

\noindent Here and throughout, $f$ denotes the projection from $M_g\times \Sigma_g$ onto $M_g$. The notation $f_!U_g$ denotes the direct image of $U_g$, which in our situation is a genuine holomorphic bundle of rank $2g-1$, with its fiber over $y\in M_g$ equal to $H^0(M_g;U_g|_y)$. We remind the reader of the Grothendieck-Riemann-Roch theorem in this setting: writing $\omega\in H^2(\Sigma_g;\Z)$ for the orientation class of the surface $\Sigma_g$, for any holomorphic vector bundle $W$ lying over $M_g\times\Sigma_g$ we have

\vspace{.25cm}
\begin{equation}
    \text{ch}\left(f_!W\right) \; = \; f_\ast\left(\text{ch}(W)\left( 1 - (g-1)\omega\right)\right). \label{eq:riemannroch}
\end{equation}
\vspace{.25cm}

\noindent From this one can obtain expressions for the Chern classes $c_i(f_!V_g)$ in terms of expressions for the Chern classes $c_i(V_g)$. Since $V_g$ itself is rank two, the only non-zero Chern classes are $c_1(V_g)$ and $c_2(V_g)$. The first of these may be written as follows:

\vspace{.25cm}
\begin{equation}
    c_1(V_g) = a_1\otimes 1 + \sum_{j=1}^{2g}b_1^j\otimes f_j + (4g-3)\otimes \omega.\label{eq:c1}
\end{equation}
\vspace{.25cm}

\noindent Here we are using the K\"{u}nneth decomposition of $H^\ast(M_g\times \Sigma_g)$, and we have fixed a symplectic basis $f_1,\ldots,f_{2g}$ of $H^1(\Sigma_g;\Z)$, such that, for $1\leqslant i\leqslant g$, we have $f_if_{i+g} =\omega$ and $f_if_j =0 $ for $j\neq i+g$. Next, the second Chern class may be written as follows:

\vspace{.25cm}
\begin{equation}
    c_2(V_g) = a_2\otimes 1 + \sum_{j=1}^{2g}b_2^j\otimes f_j + \left((2g-1)a_1-\sum_{j=1}^{g}b_1^jb_1^{j+g}\right)\otimes \omega.\label{eq:c2}
\end{equation}
\vspace{.25cm}

\noindent The terms appearing in front of $\omega$ are computed in \cite{ab}. Other than these tail terms, the expressions (\ref{eq:c1}) and (\ref{eq:c2}) implicitly define the following elements:

\vspace{.25cm}
\[
    a_1 \in H^2(M_g;\Z), \qquad a_2 \in H^4(M_g;\Z), \qquad b_1^j \in H^1(M_g;\Z), \qquad b_2^j \in H^3(M_g;\Z),
\]
\vspace{.25cm}

\noindent in which $1\leqslant j \leqslant 2g$. Next, we use the direct image bundle to define the following classes:

\vspace{.25cm}
\[
    d_i \; = \; d_{g,i} \; := \; c_i(f_! V_g) \; \in H^{2i}(M_g;\Z), \qquad 1\leqslant i \leqslant 2g-1.
\]
\vspace{.25cm}

\noindent We remark that the Riemann-Roch formula (\ref{eq:riemannroch}) implies $d_{1} = (g-1)a_1$, which may be written more explicitly as $c_1(f_!V_g) = (g-1)c_1( V_g|_{M_g})$. This is briefly explained below. As warned in the introduction to this section, in contrast to the classes $a_1,a_2,b_1^j,b_2^j$, the restriction of $d_{g,i}$ to $M_{g-1}$ is {\emph{not}} equal to $d_{g-1,i}$. This is evident for $d_1$, as just seen, and will be clear more generally from the formulas below. We now state the fundamental result due to Atiyah-Bott:

\vspace{.55cm}
\begin{theorem}[\cite{ab} Thm 9.11]\label{thm:ab}
    The elements $a_1,a_2,b_1^j,b_2^j,d_{i}$ generate the ring $H^\ast(M_g;\Z)$, where the indices run over $1\leqslant j \leqslant 2g$ and $2\leqslant i \leqslant 2g-1$.
\end{theorem}
\vspace{.55cm}

\noindent Since $d_{1}$ is an integral multiple of $a_1$, it is in fact redundant. We can also show that the generator $a_2$ is redundant for certain coefficient rings, as follows.

In principle, all the classes $d_{i}$ can be computed from the Riemann-Roch formula (\ref{eq:riemannroch}) as rational expressions in the generators $a_1,a_2,b_1^j,b_2^j$. We will, essentially, accomplish this later using a computational framework set up by Zagier. As a basic example, however, we consider the computation of $d_{2}$. First, we remind the reader of the first few terms of the Chern character:

\vspace{.18cm}
\[
	\text{ch} \; = \; \text{rk} +c_1 + \left(\frac{1}{2}c_1^2 - c_2 \right) + \left(\frac{1}{6}c_1^3 - \frac{1}{2}c_1c_2 + \frac{1}{2}c_3 \right) + \ldots
\]
\vspace{.1cm}

\noindent Here `$\text{ch}$' is the Chern character of any complex vector bundle, `$\text{rk}$' is the rank, and $c_i$ stands for the $i^\text{th}$ Chern class. To begin using Riemann-Roch on our universal bundle $V_g$, we first note that the powers of $c_1(V_g)$ and $c_2(V_g)$ are straightforward to compute:

\vspace{.28cm}
\[
	c_1(V_g)^n \; = \; a_1^n\otimes 1 + n\sum_{j=1}^{2g} a_1^{n-1} b_1^j \otimes f_j + \left( n(4g-3)a_1^{n-1} - n(n-1) a_1^{n-2} B_1 \right)\otimes \omega
\]

\[
	c_2(V_g)^n \; = \; a_2^n\otimes 1 + n\sum_{j=1}^{2g} a_2^{n-1} b_2^j \otimes f_j + \left( n(2g-1)a_1a_2^{n-1} - nB_1a_2^{n-1} + n(n-1) a_2^{n-2} B_2 \right)\otimes \omega
\]
\vspace{.25cm}

\noindent Here we have set $B_{i} = \sum_{j=1}^g b_i^jb_i^{j+g}$ for $i=1,2$. As the bundle $V_g$ has rank two, we know that all Chern classes $c_i(V_g)$ for $i\geqslant 3$ are zero. It is then a routine matter to write out the first few terms of $\text{ch}(V_g) - (g-1)\text{ch}(V_g)\cdot 1\otimes \omega$, and then apply $f_\ast$, which simply picks out the terms in this expression that factor out an $\omega$. Setting this equal to $\text{ch}(f_!V_g) = \text{rk}(f_!V_g) + d_{1} + \frac{1}{2}d_{1}^2 - d_{2} + \ldots$, as (\ref{eq:riemannroch}) dictates, yields the equalities $\text{rk}(f_!V_g) =2g-1$ and $d_{1} = (g-1)a_1$, which were mentioned above, and also

\vspace{.28cm}
\begin{equation}
    d_{2} \; = \;  \frac{1}{2}\left((g-1)(g-2) a_1^2 + (2g-1)a_2 + a_1B_1 - B_{12}\right)\label{eq:d2}
\end{equation}
\vspace{.25cm}

\noindent where $B_{12} = \sum_{j=1}^g b_1^j b_2^{j+g} - b_1^{j+g}b_2^j$. From this equation we see that $a_2$, multiplied by the number $2g-1$, is equal to an integral expression in the generators $a_1,b_1^j,b_2^j$ and $d_{2}$. Thus:

\vspace{.65cm}
\begin{corollary}\label{prop:powertwo}
    If $m$ and $2g-1$ are coprime, then the residue classes of the elements $a_1,b_1^j,b_2^j,d_{i}$ generate the ring $H^\ast(M_g;\Z_m)$, where the indices run over $1\leqslant j \leqslant 2g$ and $2\leqslant i \leqslant 2g-1$.
\end{corollary}
\vspace{.55cm}

\noindent Finally, we take a moment to mention an elementary but important point. Recall that the cohomology ring of any space is a graded commutative ring. This means that $ab= (-1)^{|a||b|}ba$ for any two homogeneously graded elements $a,b$ in the ring, where $|a|$ denotes the grading of $a$. When we take the tensor product of two such rings, the product is the graded commutative product, given by

\vspace{.18cm}
\[
    (a\otimes b)\cdot (c\otimes d) \; = \;  (-1)^{|b||c|}ac\otimes bd.
\]
\vspace{.1cm}

\noindent This is relevant in the above computations, all done in the context of a K\"{u}nneth decomposition, and the reader should be aware of this for the computations below. When all elements involved are of even gradings, as is often the case, there is of course no difference between this product and the ordinary product induced by tensor product.

\vspace{.65cm}

\subsection{The redundancy of some generators over $\Z_p$}

Here we explain why some of the $d_i$ are redundant generators when working over the field $\Z/p$ for $p$ prime. We begin by recalling where Atiyah-Bott's generators for $H^\ast(M_g;\Z)$ come from.\\

Recall that $M_g$ may be identified with the moduli space of projectively flat connections on a $U(2)$-bundle $P$ over $\Sigma_g$ with odd first Chern class. Let $\mathscr{G}$ be the gauge group consisting of bundle automorphisms of $P$, and write $\overline{\mathscr{G}}$ for the quotient of $\mathscr{G}$ by its constant central $U(1)$-subgroup. Further, write $\mathscr{C}$ for the affine space of connections on $P$, and $\mathscr{C}_{ss}$ for stratum of projectively flat connections. From the holomorphic viewpoint, this is the semi-stable stratum. Atiyah-Bott show that there is an induced surjection in equivariant cohomology

\vspace{.3cm}

\[
    H^\ast_{\overline{\mathscr{G}}}(\mathscr{C};\Z)\;\; \longrightarrow \;\; H^\ast_{\overline{\mathscr{G}}}(\mathscr{C}_{ss};\Z).
\]

\vspace{.3cm}

\noindent Indeed, they show that the Yang-Mills functional on $\mathscr{C}$ is equivariantly perfect Morse-Bott, and $\mathscr{C}_{ss}$ is the manifold of absolute minima. The domain of this map may be identified with the ordinary cohomology of the classifying space $B\overline{\mathscr{G}}$ and the codomain with that of $M_g$. They then obtain the generators for the cohomology of $M_g$ from generators for that of $B\overline{\mathscr{G}}$. The generators for the cohomology of $B\overline{\mathscr{G}}$ are obtained via the homological triviality of the following three fibrations:

\vspace{.3cm}

\begin{center}
    \begin{tikzcd}
    \Omega U(2)  \arrow[r] &  B\mathscr{G}^\# \arrow[d] & & B\mathscr{G}^\# \arrow[r] & B\mathscr{G} \arrow[d] &  & BU(1) \arrow[r] & B\mathscr{G} \arrow[d]\\
        & U(2)^{2g} & & & BU(2) && & B\overline{\mathscr{G}} 
    \end{tikzcd}
\end{center}

\vspace{.3cm}

\noindent We have written $\mathscr{G}^\#$ for the based gauge group and $\Omega U(2)$ for the identity component of the based loop space of $U(2)$. The classes $a_1$ and $a_2$ come from generators for the cohomology of $BU(2)$, while $b_1^j$ and $b_2^j$ correspond to generators of the $j^\text{th}$ factor of $U(2)$ inside the product $U(2)^{2g}$. The generators $d_i$ are replaceable by generators $e_i$ which can be traced back to generators for the cohomology of $\Omega U(2)$. See Proposition 2.20 of \cite{ab} for more details.\\

Recalling that the loop space of a circle is homotopy equivalent to a countable set of points, and that $U(2)$ is topologically a circle times a 3-sphere, we conclude that the based loop space of $U(2)$ is homotopy equivalent to $\Z\times \Omega S^3$, and in particular $\Omega U(2)$ may be identified with the loop space of the 3-sphere. Now, the cohomology ring of $\Omega S^3$ is well-known to be isomorphic to a {\emph{divided polynomial algebra}}. Recall \cite{hatcher} that the divided polynomial algebra $\Gamma_\Z[x]$ at level $n$, for some even integer $n=\text{deg}(x_1)$, is a ring with generators $x_i$ for $i\geqslant 1$ with deg$(x_i)=ni$ such that $x_1^k = k! x_k$. Consequently, $x_i x_j = {i+j \choose i}x_{i+j}$. Note that as a rational algebra, $\Gamma_\mathbb{Z}[x]\otimes \Q$ is generated by $x_1$. The cohomology ring of $\Omega S^3$ is isomorphic to $\Gamma_\Z[x]$ with $\text{deg}(x_1)=2$.\\

For a prime number $p$, the divided polynomial algebra $\Gamma_{\Z_p}[x]$ over the field $\Z_p$ does not need nearly as many generators. In fact, see for example loc. cit., we have an isomorphism

\vspace{.3cm}

\[
    \Gamma_{\Z_p}[x] \; \cong \; \bigotimes_{i\geqslant 0} \Z_p[x_{p^i}]/(x_{p^i}^p).
\]

\vspace{.3cm}

\noindent Since the lifts of the generators $d_i$ in the cohomology of $B\overline{\mathscr{G}}$ as in \cite{ab} come from generators for the cohomology of $\Omega U(2)$ via the homological triviality of the above fibrations, over $\Z_p$ one only needs the mod $p$ residue classes of the generators $d_{p^i}$. Thus:

\vspace{.65cm}
\begin{corollary}\label{prop:powertwo}
    If $p$ is prime, then the residue classes of the elements $a_1,a_2,b_1^j,b_2^j,d_{p^i}$ generate the ring $H^\ast(M_g;\Z_p)$, where the indices run over $1\leqslant j \leqslant 2g$ and $2\leqslant p^i \leqslant 2g-1$.
\end{corollary}
\vspace{.25cm}

\vspace{.25cm}

\subsection{Integral generators for the cohomology of $N_g$}\label{sec:ng}

We now proceed to the generators of the fixed determinant moduli space $N_g$. Using the K\"{u}nneth decomposition of $H^\ast(N_g\times \Sigma_g)$, we implicitly define

\vspace{.18cm}
\[
    \alpha \in H^2(N_g;\Z), \qquad  \psi_j\in H^3(N_g;\Z), \qquad \beta\in H^4(N_g;\Z),
\]
\vspace{.1cm}

\noindent in which $1\leqslant j \leqslant 2g$, by the following Chern class expression, with constants arranged to follow the standard conventions in the literature:

\vspace{.18cm}
\begin{equation}
    c_2\left(\text{End}(V_g)|_{N_g\times\Sigma_g}\right) \; = \;  -\beta\otimes 1  + 4\sum_{j=1}^{2g}\psi_j\otimes f_j+ 2\alpha\otimes\omega.\label{eq:c2end}
\end{equation}
\vspace{.1cm}

\noindent We will shortly relate these classes to the generators of $M_g$ mentioned thus far. For this we will use the embedding $\iota:N_g\longrightarrow M_g$. It will be useful for the sequel to consider how the intersection pairings for $N_g$ and $M_g$ are related, and for this we make use of a $4^g$-fold covering map

\vspace{.18cm}
\begin{equation}
    p:N_g\times J_g \xrightarrow{4^g:1} M_g,\label{eq:cov}
\end{equation}
\vspace{.1cm}

\noindent in which $J_g$ is the Jacobian torus of $\Sigma_g$. More precisely, $M_g$ is the quotient of $N_g\times J_g$ by a free $\Z_2^{2g}$-action. The Jacobian is the moduli space of flat $U(1)$ connection on $\Sigma_g$, and this covering is defined by tensoring the connection classes in $N_g$ and $J_g$.

The map $p$ induces an isomorphism in rational cohomology, as is shown in \cite[Sec. 9]{ab}. This may be deduced from the fact that the relevant $\Z_2^{2g}$-action on $H^\ast(N_g\times J_g;\Q)$ is trivial. We now write down the effect of $p^\ast$ on some of the generators that we have introduced thus far. Let $\theta_j\in H^1(J_g;\Z)$ be the generator corresponding to $f_j\in H^1(\Sigma_g;\Z)$. Then we have:

\vspace{.55cm}
\begin{prop}\label{prop:pmap}
    The homomorphism $p^\ast:H^\ast(M_g;\Z) \longrightarrow H^\ast(N_g;\Z)\otimes H^\ast(J_g;\Z)$ is given by:
    \[
        p^\ast(a_1) \; = \; \alpha\otimes 1 + 1\otimes 4\Theta, \qquad p^\ast(a_2) \; = \; \frac{1}{4}\left(p^\ast(a_1)^2-\beta\otimes 1\right)
    \]
    \[
        p^\ast(b_1^j) \; = \; 1\otimes 2\theta_j,\qquad\;\;\;\;\; p^\ast(b_2^j) \; = \; \psi_j\otimes 1 + p^\ast(a_1)\cdot \left(1\otimes \theta_j\right)
    \]
\vspace{.10cm}

\noindent in which $\Theta = \sum_{j=1}^{g} \theta_j\theta_{j+g}$. 
\end{prop}
\vspace{.55cm}

\begin{proof}
The proof of this proposition is more or less implicit in Atiyah-Bott \cite[p. 585]{ab}; we briefly sketch the argument. We first recall the identity $c_2(\text{End}(W))=4c_2(W)-c_1(W)^2$ for any rank two bundle $W$. Letting $\iota:N_g\longrightarrow M_g$ denote the inclusion map, we then equate the terms of (\ref{eq:c2end}) with $\iota^\ast$ applied to the expression $4c_2(V_g)-c_1^2(V_g)$ formed by (\ref{eq:c1}) and (\ref{eq:c2}) to obtain:

\vspace{.3cm}
\begin{equation}\label{eq:iota}
    \iota^\ast\left(a_1-B_1\right) \; = \; \alpha, \qquad \iota^\ast\left(b_2^j -a_1b_1^j/2\right) \; = \; \psi_j,\qquad \iota^\ast\left(a_1^2 - 4a_2\right) \; = \; \beta.
\end{equation}
\vspace{.3cm}

\noindent Next, observe that the endomorphism bundle of $V_g$ is acted on trivially by $J_g$, and that the restriction of $p$ to $N_g$ is equal to $\iota$. These observations imply the equations obtained from (\ref{eq:iota}) by replacing each $\iota^\ast$ with $p^\ast$, and replacing $\alpha$ with $\alpha\otimes 1$, and similarly for $\psi_j$ and $\beta$. Otherwise said, the pullback of $V_g$ via the map $p$ factors through $\iota$. The relations of the resulting equations determine the proposition, except for the fact that $p^\ast(b_1^j)=1\otimes 2\theta_j$. This last point has only to do with how the 1-skeleton of $N_g\times J_g$ is mapped to $M_g$ under $p$, which, at least up to sign, is transparent from the covering structure: each loop upstairs double covers a loop downstairs. To be more precise, we note that $V_g$ pulls back and restricts over $J_g\times \Sigma_g$ to the bundle $U_{J}^{\otimes 2}$ in which $U_{J}$ is a universal bundle over $J_g\times \Sigma_g$. One can then compute that $c_1(U_J) = \sum \theta_j\otimes f_j$, see for example \cite[Lemma 2.23]{fl}, which in turn implies $p^\ast(b_1^j)=1\otimes 2\theta_j$.
\end{proof}

\vspace{.3cm}

\noindent Observe that this proposition completely determines the map $p^\ast$, since the generators under consideration rationally generate the cohomology ring of $M_g$. From the above computation, we gather that the homomorphism $\iota^\ast:H^\ast(M_g;\Z)\longrightarrow H^\ast(N_g;\Z)$ is determined as follows:

\vspace{.30cm}

\[
    \iota^\ast(a_1) \; = \; \alpha, \qquad \iota^\ast(a_2) \;=\; (\alpha^2-\beta)/4, \qquad \iota^\ast(b_1^j) \; = \; 0, \qquad \iota^\ast(b_2^j) \; = \; \psi_j.
\]

\vspace{.30cm}

\noindent Now, with an eye towards producing generators for the integral cohomology ring of $N_g$, we define the $\delta_i$ from the introduction to be the restrictions of the classes $d_i$ from $M_g$:

\vspace{.3cm}
\[
    \delta_ i \; = \; \delta_{g,i} \; := \;  \iota^\ast(d_i) \; = \; c_i\left(f_!V_g|_{N_g}\right).
\]
\vspace{.3cm}

\begin{prop}\label{prop:intng}
    The elements $\alpha,\,\frac{1}{4}(\alpha^2-\beta),\,\psi_j,\,\delta_i$ generate the ring $H^\ast(N_g;\Z)$, where the indices run over $1\leqslant j \leqslant 2g$ and $2\leqslant i \leqslant 2g-1$. 
\end{prop}

\vspace{.4cm}

\begin{proof}
    Since these classes are the images of the generators for $M_g$ under $\iota^\ast$, it suffices to show that $\iota^\ast$ is onto. For this we consider the map $M_g \longrightarrow J_g$ which sends a connection class to its determinant connection class. This is a fibration with fibers homeomorphic to $N_g$. The Leray-Serre spectral sequence for this fibration collapses at the $E_2$-page: any non-trivial differentials, or non-trivial local-coefficient systems, are ruled out by the fact that the cohomologies of $M_g$ and $N_g\times J_g$ are torsion-free and of the same rank. The collapsing at $E_2$ then implies that the restriction map from the cohomology of $M_g$ to that of $N_g$ is surjective.
\end{proof}

\vspace{.5cm}

\begin{corollary}\label{cor:nggens}
    If $p$ is prime, then the residue classes of the elements $\alpha,\,\frac{1}{4}(\alpha^2-\beta),\,\psi_j,\,\delta_{p^i}$ generate the ring $H^\ast(N_g;\Z_p)$, where $1\leqslant j \leqslant 2g$ and $2\leqslant p^i \leqslant 2g-1$. If $p\nmid 2g-1$, then $\frac{1}{4}(\alpha^2-\beta)$ is redundant.
\end{corollary}

\vspace{.65cm}

\subsection{Twisting by a line bundle to define $z_i$ and $\xi_i$}

We now describe how the generators $d_i$ and $\delta_i$ can be replaced with generators obtained from twisting by a line bundle, and then define the classes $z_i$ and $\xi_i$.\\

The elements $d_{i}$ were defined as the Chern classes $c_i(f_!V_g)$ in which the universal bundle $V_g$ is normalized such that $c_1(f_!V_g) = (g-1)c_1(V_g|_{M_g})$, i.e. $d_{1}=(g-1)a_1$. However, Theorem \ref{thm:ab} as stated by Atiyah-Bott holds for any normalization of $V_g$. In particular, if we consider the universal bundle which is $V_g$ twisted by a power of the pull-back of $\det(V_g|_{M_g})$, i.e. the bundle

\vspace{.3cm}
\begin{equation}
    V_g\otimes  f^\ast \det\left(V_g|_{M_g}\right) ^{\otimes n}\label{eq:bundletwist}
\end{equation}
\vspace{.3cm}

\noindent then the Chern classes of its direct image will still, along with $a_1,a_2,b_1^j,b_2^j$, generate the integral cohomology ring of $M_g$.  When we consider the direct image of (\ref{eq:bundletwist}) under the projection $f$, it is useful to mention that in general $f_! \left(W \otimes f^\ast L\right)$ is isomorphic to $f_!W \otimes L$. Here we recall the Chern class formula for tensoring a vector bundle $W$ of rank $r$ by a line bundle $L$:

\vspace{.3cm}
\begin{equation}
   c_i(W\otimes L) \; = \; \sum_{j=0}^{i} {r-j\choose i-j}c_1(L)^{i-j}c_j(W).\label{eq:linebundle}
\end{equation}
\vspace{.3cm}

\noindent This tells us how the generators $d_i$ transform after twisting by a line bundle: upon setting $W=f_!V_g$ and $L=\det(V_g|_{M_g})$, the above formula has $r=2g-1$, $c_j(W)=d_j$ and $c_1(L) = (2n-1)a_1$. Although setting $n=-1/2$ does {\emph{not}} transform the $d_i$ to {\emph{integral}} generators, it is a case of particular computational interest to us, and so we define the transformed generators:

\vspace{.3cm}
\begin{equation*}
   z_i \; := \; c_i\left(f_!V_g\otimes \text{det}\left(V_g|_{M_g}\right)^{-1/2}\right), \qquad \xi_i \; := \; c_i\left(f_!V_g\otimes \text{det}\left(V_g|_{M_g}\right)^{-1/2}\right).
\end{equation*}
\vspace{.3cm}

\noindent Of course, the bundles here are not actual vector bundles, but $z_i$ and $\xi_i$ may be defined in terms of $d_i$ and $\delta_i$ using (\ref{eq:linebundle}). Alternatively, one may think of the bundles that appear in the setting of rational $K$-theory. The class $2^i z_i$ (resp. $2^i \xi_i$) is in the integral cohomology of $M_g$ (resp. $N_g$). The formula (\ref{eq:linebundle}) relating $\delta_i$ with $\xi_i$ is what appears in the introduction as (\ref{eq:defxi}), and a similar formula holds replacing $\xi_i$ with $z_i$, $\delta_i$ with $d_i$, and $\alpha$ with $a_1$. Note $\xi_1 = -\alpha/2$ and $z_1 = -a_1/2$. For the same reason as was for $d_i$ and $\delta_i$, the classes $z_i$ and $\xi_i$ may sometimes be written $z_{g,i}$ and $\xi_{g,i}$.\\

Finally, we mention that if we are working over the coefficient ring $\Z_m$ with $m$ odd, then $\xi_i$ may be defined as an honest class in $H^{2i}(N_g;\Z_m)$ by interpreting $1/2$ in the above formulas as the mod $m$ inverse of $2$. A similar remark holds for the classes $z_i$. We then have:

\vspace{.5cm}

\begin{prop}\label{prop:xigens}
    If $m$ is odd, to generate the ring $H^\ast(N_g;\Z_m)$, we may replace the $\delta_i$ generators by the $\xi_i$ classes as interpreted above. Similarly, to generate $H^\ast(M_g;\Z_m)$ we may replace $d_i$ by $z_i$.
\end{prop}

\vspace{.65cm}

\subsection{Generators for the invariant subring of $N_g$}

The mapping class group of $\Sigma_g$ acts on the moduli space $N_g$ in a natural way. The subgroup of the mapping class group that acts trivially on the homology of $\Sigma_g$, called the Torelli group, acts trivially on $H^\ast(N_g;\Z)$. Thus the action of the mapping class group on $H^\ast(N_g;\Z)$ descends to an action of the quotient group, which is $\text{Sp}(H^1(\Sigma_g;\Z))$. Having previously chosen a symplectic basis for the first cohomology group of $\Sigma_g$, we may identify this as an action of $\text{Sp}(2g,\Z)$.\\

The classes $\alpha$ and $\beta$ are invariant under this action, while the classes $\psi_j$ behave under the action as does a standard symplectic basis. It is convention to define the degree 6 element

\vspace{.3cm}

\[
    \gamma \; := \; -2 \sum_{j=1}^{g} \gamma_j \;\; \in H^6(N_g;\Z),\qquad \gamma_j \; := \; \psi_j \psi_{j+g},
\]

\vspace{.3cm}

\noindent for then $\alpha,\beta,\gamma$ generate the invariant ring over the rationals. This is a basic exercise in $\text{Sp}(2g,\Z)$-representation theory: the free graded-commutative algebra generated by the $\psi_j$ has its invariant ring over the rationals generated by $\gamma$. Over the integers, however, the invariant ring is a divided polynomial algebra $\Gamma_\Z[\upsilon]$, in which we define $\upsilon_k$ for $k\geqslant 1$ as follows:

\vspace{.3cm}

\[
    \upsilon_k \; := \; \sum_{i_1 <  \ldots < i_k} \gamma_{i_1}\cdots\gamma_{i_k} \; = \; \gamma^k/2^k k!
\]

\vspace{.3cm}

\noindent We learned earlier that when working over the field $\Z_p$ for $p$ prime, one only needs the generators $\upsilon_{p^i}$. Note that $\upsilon_k=0$ for $k\geqslant g$ for degree reasons. The classes $\delta_i$ as well as $\xi_i$ are invariant under the action for the same reasons as are $\alpha$ and $\beta$; alternatively, we will later see explicit expressions for these classes as rational polynomials in $\alpha,\beta,\gamma$. We conclude:

\vspace{.4cm}

\begin{prop}\label{prop:inv}
${}$\\
\vspace{-.1cm}
\begin{enumerate}
    \item[(i)] The $\text{{\emph{Sp}}}(2g,\Z)$-invariant subring of $H^\ast(N_g;\Z)$ is generated by $\alpha,\,\frac{1}{4}(\alpha^2-\beta),\,\delta_i,\,\upsilon_k$ where the indices run over $2\leqslant i \leqslant 2g-1$ and $1\leqslant k < g$. 
    \item[(ii)] For $p$ prime, the $\text{{\emph{Sp}}}(2g,\Z)$-invariant subring of $H^\ast(N_g;\Z_p)$ is generated by $\alpha,\,\frac{1}{4}(\alpha^2-\beta),\,\delta_{p^i},\,\upsilon_{p^k}$ where the indices run over $2\leqslant p^i \leqslant 2g-1$ and $1\leqslant p^k < g$. 
\end{enumerate}
\end{prop}

\vspace{.4cm}

\noindent As seen earlier, if in (ii) we have $p \nmid 2g-1$, then $(\alpha^2-\beta)/4$ is redundant. Further, if $p$ is odd, then the $\delta_i$ generators can be replaced by the $\xi_i$ just as in the previous subsection. A similar proposition may be crafted for the invariant subring of $M_g$, for which one has the classes $b_1^j$ and $b_2^j$ instead of just the $\psi_j$, but we will not pursue this.

\vspace{.65cm}

\vspace{.35cm}

\section{The intersection pairings for Newstead classes}\label{sec:intersection}

The computational framework of Zagier \cite{zagier} that we use to prove Theorem \ref{thm:maincomp} is derived from intersection pairing formulas of Thaddeus \cite{thaddeus-conformal} for monomials in $\alpha,\beta,\psi_j$. We will not work directly with these formulas, but will need some of their properties for later.\\

Recall that $\dim N_g = 6g-6$ and $\text{deg}(\alpha)=2$ and $\text{deg}(\beta)=4$. Thaddeus computes

\vspace{.18cm}
\begin{equation}\label{eq:alphabeta}
    \alpha^i\beta^j[N_g]  = (-1)^{g}\frac{i!}{(i-g+1)!}2^{2g-2}\left(2^{i-g+1}-2\right)B_{i-g+1}
\end{equation}
\vspace{.1cm}

\noindent whenever $i+2j=3g-3$, where $B_n$ is the $n^{\text{th}}$ Bernoulli number, and should not be confused with the elements $B_1$ and $B_2$ defined earlier. This formula is the most fundamental; the inclusion of the $\psi_j$ terms is handled with the following genus recursive property. For any subset $K\subset \{1,\ldots,g\}$ with cardinality $|K|=k$, and any $j\geqslant 0$ such that $i+3k+2j=3g-3$, we have

\vspace{.3cm}
\begin{equation}
    \alpha^i\beta^j\prod_{\ell\in K}\psi_\ell\psi_{\ell+g}[N_g]  = \pm \alpha^i\beta^j[N_{g-k}].\label{eq:psi}
\end{equation}
\vspace{.3cm}

\noindent The sign, which is $(-1)^{|K|}$, is not important for us. To prove (\ref{eq:psi}), Thaddeus shows that $\psi_j\psi_{j+g}$ is Poincar\'{e} dual to the homology class of an appropriately embedded $N_{g-1}$ inside $N_g$. Thaddeus further shows that if a subset $K\subset\{1,\ldots,2g\}$ is such that $K\neq K+g$, then any pairing $\alpha^i\beta^j\prod_{\ell \in K} \psi_\ell [N_g]$ vanishes. Here $K+g$ is the set of elements $k+g$ where $k\in K$, in which addition is understood mod $2g$.\\

We now derive some similar properties for the intersection pairings of the larger moduli space $M_g$. The pairings in $H^\ast(M_g;\Z)$ can be understood in terms of those in $N_g$ using the covering (\ref{eq:cov}) from $N_g\times J_g$ down to $M_g$: if $x\in H^{8g-6}(M_g;\Z)$ is a top degree element, then

\vspace{.3cm}
\begin{equation}
    x[M_g] \; = \; \frac{1}{4^{g}}\left(p^\ast(x)\, \slash \, \Omega_g \right)[N_g],\label{eq:cup}
\end{equation}
\vspace{.3cm}

\noindent where we have taken the slant-product with $\Omega_g \in H^{2g}(J_g;\Z)$, the orientation class of the Jacobian $J_g$. The factor of $1/4^{g}$ appears because $p$ is a $4^{g}$-fold covering. From Prop. \ref{prop:pmap} and (\ref{eq:alphabeta}) we compute

\vspace{.3cm}
\begin{equation}\label{eq:a1}
    a_1^{4g-3}[M_g]  = \frac{(4g-3)!}{(2g-2)!}2^{2g-2}\left(2^{2g-2}-2\right)|B_{2g-2}|
\end{equation}
\vspace{.3cm}

\noindent Indeed, this is the result of expanding $(\alpha\otimes 1+4\otimes \Theta)^{4g-3}/4^g$, taking the slant product with $\Omega_g$, which picks out the term in front of $\Theta^g/g!$, and evaluating against $[N_g]$. One can similarly use Prop. \ref{prop:pmap} and (\ref{eq:alphabeta}) to compute intersection pairings in $M_g$ for monomials in $a_1,a_2$. Next, we have the following analogue for $M_g$ of the vanishing property for the $\psi_j$ classes:

\vspace{.65cm}
\begin{prop}\label{prop:bterms}
    Suppose $J_1$ and $J_2$ are subsets of $\{1,\ldots,2g\}$, and that $x\in H^\ast(M_g;\Z)$ is an element invariant under the $\text{{\emph{Sp}}}(2g,\Z)$-action. Then with $I_i = J_i\cap(J_i+g)$ for $i=1,2$ we have
    
    \vspace{.25cm}
    \[
        J_1\setminus I_1 \neq (J_2\setminus I_2)+g \quad \implies \quad x\prod_{j\in J_1}b_1^j \prod_{j\in J_2}b_2^j [M_g] \; = \; 0.
    \]    
\end{prop}
\vspace{.35cm}

\begin{proof} We adapt Thaddeus's argument \cite{thaddeus-conformal}, and use that $\text{Sp}(2g,\Z)$ acts in the same standard way on $\{b_1^j\}$ and $\{b_2^j\}$. First, suppose $k\in J_1\setminus I_1$ and either $k,k+g$ are both in $J_2$ or both not contained in $J_2$. Take an orientation-preserving diffeomorphism $f$ of $\Sigma_g$ such that the induced action $f^\ast$ on $H^\ast(M_g;\Z)$ fixes $b_i^j$ for $j\notin\{k,k+g\}$ while $f^\ast b_i^k =-b_i^{k}$ and $f^\ast b_i^{k+g} =-b_i^{k+g}$. Then

\vspace{.25cm}
\[
    x\prod_{j\in J_1}b_1^j\prod_{j\in J_2}b_2^j [M_g] \; = \; x\prod_{j\in J_1}f^\ast b_1^j \prod_{j\in J_2}f^\ast b_2^j [M_g]
\]
\vspace{.18cm}

\noindent where we have used invariance of the pairing. The right side, by our choice of $f^\ast$ and our hypothesis on $k\in J_1$, is equal to minus the left side, forcing the pairing to be zero. The remaining case is when $k\in (J_1\setminus I_1)\cap(J_2\setminus I_2)$. Note $p^\ast(b_1^kb_2^k)=2\psi_k\otimes\theta_k$. The vanishing then follows via (\ref{eq:cup}) from the vanishing condition for the $\psi_k$, since the only way to produce $\psi_{k+g}$ via $p^\ast$ is to include $b_2^{k+g}$.
\end{proof}

\vspace{.65cm}

\noindent We can also derive an analogue of (\ref{eq:psi}) for $M_g$ using Proposition \ref{prop:pmap} and formula (\ref{eq:cup}):

\vspace{.3cm}
\begin{equation}
    x\,b_1^jb_1^{j+g}b_2^jb_2^{j+g}[M_{g}] \; = \;  \pm i^\ast x [M_{g-1}]\label{eq:mrec}
\end{equation}
\vspace{.2cm}

\noindent Here $x$ is any element of $H^\ast(M_g;\Z)$, and $i$ is the embedding of $M_{g-1}$ into $M_g$ corresponding to collapsing the $j^{\text{th}}$ handle of $\Sigma_g$.

\vspace{.65cm}

\newpage

\section{The computation of integral intersection pairings}\label{sec:maincomp}

Here we present the main computation of the paper: we prove a generalization of Theorem \ref{thm:maincomp} and its analogue for $M_g$. The proofs rely on the work of Zagier \cite{zagier}. We will use very basic symmetric function theory, background for which can be found in Appendix \ref{appendix}.\\

First, some convenient notation. Let $E$ be a complex vector bundle over an oriented, closed manifold $M$ with $\dim M$ even. For later use, we define the Chern class polynomial $c(E)_x$ to be:

\vspace{.2cm}
\[
    c(E)_x \; = \; \sum_{i\geqslant 0} c_i(E)x^i \; \in H^\ast(M;\Z)[x].
\]
\vspace{.2cm}

\noindent For a partition $\lambda = (\lambda_1,\ldots,\lambda_k)$ we write $c_\lambda(E)$ for the product $c_{\lambda_1}(E)\cdots c_{\lambda_k}(E)$. We define the {\emph{Chern number polynomial}} of the vector bundle $E$, written $\cn(E)$, by the formula

\vspace{.2cm}
\[
    \cn(E) \; = \; \sum_{\lambda} c_{\lambda}(E)[M]\cdot m_\lambda
\]
\vspace{.2cm}

\noindent in which the sum is over all partitions $\lambda$. Here $m_\lambda$ is the monomial symmetric function associated to $\lambda$. The only partitions $\lambda$ contributing nonzero terms are those with $|\lambda|=\dim(M)/2$. Thus the Chern number polynomial is a symmetric function in the variables $x_1,x_2,\ldots$ homogeneous of degree $\dim(M)/2$. It records all of the Chern numbers of the bundle $E$ over $M$.\\

In addition to $Q(T)$ and $U(T)$ from the introduction, define the following formal power series in the variable $T$ whose coefficients are in the ring $\Lambda$ of symmetric functions with integer coefficients:

\vspace{.3cm}
\[
    R(T) \; = \;  \sum_{i\geqslant 0} (-1)^i m_{(2^{i})} T^i, \qquad P(T) \; = \; \sum_{i\geqslant 0} (-1)^i \left(2 m_{( 2^i 1^2)} + \left(i+1\right)^2 m_{(2^{i+1})}\right)T^i.
\]
\vspace{.3cm}

\noindent Now, recall that for $k\leqslant g$ we have an embedding of the lower genus moduli space $M_k$ into $M_g$, and similarly of $N_k$ into $N_g$. The particular choice of embedding is not important. \\

\vspace{.5cm}

\begin{theorem}\label{thm:maincomputation} Let $Z_g = f_!V_{g}\otimes \text{{\emph{det}}}\left(V_g|_{M_g}\right)^{{\small{-1/2}}}$ be the virtual bundle with $c_i(Z_g)=z_i$. For $g\geqslant k\geqslant 1$:\\
\begin{align}
    \cn\left(Z_g|_{M_{k}}\right)  \; & = \;  \frac{(-1)^{k}}{2^{2k-1}}\cdot \underset{T^{k-1}}{\text{{\emph{Coeff}}}} \Big[P(T)^k R(T)^{g-k}Q(T)^{-1}\Big] \label{eq:mg}\\
    \nonumber {}&\\
    \cn\left(Z_g|_{N_{k}}\right)  \; &= \;  \;\;\,\frac{1}{2^{k-1}}\cdot \underset{T^{k-1}}{\text{{\emph{Coeff}}}} \Big[U(T)^k R(T)^{g-k}Q(T)^{-1}\Big]\label{eq:ng}
\end{align}
\end{theorem}

\vspace{.85cm}

\noindent Recall from the introduction that $1/Q(T)$ does not quite have coefficients in the ring $\Lambda$ of symmetric functions: its coefficient in front of $T^i$ has a factor of $1/e_1^{i+1}$. However, since the constant coefficients of $R(T)$ and $P(T)$ are respectively $1$ and $e_1^2$, the formal power series inside the brackets of (\ref{eq:mg}) has $\Lambda$ coefficients in front of $T^i$ for $0 \leqslant i \leqslant k-1$. A similar remark holds for (\ref{eq:ng}).\\

Before proceeding to the proof, we explain how this theorem completely determines all of the integral intersection pairings in the cohomology ring of $N_g$, and most of the pairings for that of $M_g$. First of all, from the definitions, the left-hand side of (\ref{eq:ng}) is equal to

\vspace{.3cm}
\begin{equation}
    \sum_{\lambda} \xi_{g,\lambda_1}\xi_{g,\lambda_2}\cdots \xi_{g,\lambda_n}[N_k]\cdot m_{\lambda}\label{eq:cnxi}
\end{equation}
\vspace{.3cm}

\noindent and so when $g=k$, we obtain Theorem \ref{thm:maincomp} of the introduction. In this expression, we conflate $\xi_{g,i}$ with its restriction to $N_k$. On the other hand, as explained in Section \ref{sec:intersection}, the pairing $\xi_{g,\lambda_1}\cdots \xi_{g,\lambda_n}[N_k]$ is equal to $\xi_{g,\lambda_1}\cdots \xi_{g,\lambda_n} \prod_{j\in J}\psi_{j}\psi_{j+g}[N_g]$ for any subset $J\subset \{1,\ldots,g\}$ with $|J|=g-k$. With (\ref{eq:defxi}), this determines all pairings on $N_g$ for monomials involving the classes $\delta_i$ and $\psi_j$. Since $\alpha = \delta_1/(g-1)$ and $(\alpha^2-\beta)/4$ can be written in terms of $\alpha$ and $\delta_2$ by restricting (\ref{eq:d2}) to $N_g$, we get all pairings in monomials involving all the integral generators for the cohomology of $N_g$ in Proposition \ref{prop:intng}.\\

The situation for $M_g$ is quite similar, except certain pairings, such as $b_1^jb_2^{j+g}z_{\lambda_1}\cdots z_{\lambda_n}[M_g]$, are not covered by (\ref{eq:mrec}) in conjunction with (\ref{eq:mg}), and are not shown to vanish by (\ref{prop:bterms}). For the proof of Theorem \ref{thm:nilp} later, we will handle these pairings in a less direct way.

\vspace{.6cm}

\begin{proof}[Proof of Theorem \ref{thm:maincomputation}]
We first prove (\ref{eq:mg}). We will perform the computation by passing from $M_k$ to the covering space $N_k\times J_k$ via (\ref{eq:cov}). Define the following product of Chern polynomials: 

\vspace{.3cm}
\[
    F(x_1,x_2,\ldots) \; := \; \prod_{\ell\geqslant 1} p^\ast c( Z_{g})_{-2x_\ell}.
\]
\vspace{.3cm}

\noindent where $p^\ast Z_g$ is the pulled back bundle over $N_k\times J_k$. From the pairing formula (\ref{eq:cup}), and keeping note of the factors of $-2$ in the variables of the Chern polynomials, we have

\vspace{.3cm}
\[
    \left(F(x_1,x_2,\ldots)\, / \, \Omega_k \right)[N_k] \; = \; (-2)^{4k-3}\cdot \cn\left(p^\ast \left(Z_g|_{M_k}\right)\right) \; = \;  (-2)^{4k-3} 4^{k} \cdot\cn(Z_g|_{M_k}).
\]
\vspace{.3cm}

\noindent As usual, the factor $4^k$ accounts for the number of sheets of the covering $p$. On the other hand, we can give an explicit formula for the product of Chern polynomials. Henceforth, we will write $\beta$ instead of $\beta\otimes 1$, and so on, omitting the tensor notation from elements in the K\"{u}nneth decomposed cohomology $H^\ast(N_k\times J_k;\Q)$. Then, according to Zagier \cite[Eq. 29]{zagier}:

\vspace{.3cm}
\begin{equation}
    p^\ast c(Z_g)_{-2x} \; = \; \left(1-\beta x^2\right)^{g-1/2}\left(\frac{1+\sqrt{\beta} x}{1-\sqrt{\beta} x}\right)^{\gamma^\ast/2\beta\sqrt{\beta}}\text{exp}\left(\frac{ 4\Theta x  + 4 \Xi  x^2 - 2\gamma x/\beta}{1-\beta x^2}\right)\label{eq:zagchern}
\end{equation}
\vspace{.3cm}

\noindent in which $\Xi = \sum_{j=0}^i \psi_j\otimes \theta_{j+g} - \psi_{j+g}\otimes \theta_j$ and $\gamma^\ast = 2\gamma + \alpha\beta$. Zagier actually considers the direct image of a universal bundle over $N_k\times J_k$, rather than taking the direct image on $M_k$ and then pulling back. This is why (\ref{eq:zagchern}) has $4$'s in front of $\Theta$ and $\Xi$. We then compute the product to be

\vspace{.2cm}
\begin{equation}\label{eq:F}
    F(x_1,x_2,\ldots) \; = \; u_0^{g-1/2}\text{exp}\left((u_3-u_1)\gamma^\ast/\beta  + u_1\alpha  + 4 u_1\Theta + 4 u_2\Xi \right),
\end{equation}
\vspace{.2cm}

\noindent with the terms $u_m=u_m(\beta)$ for $0\leqslant m \leqslant 3$, which are formal power series in $\beta$ with coefficients that are symmetric polynomials in the variables $x_\ell$, defined as follows:

\vspace{.2cm}
\[
    u_0 \; := \; \prod_{\ell\geqslant 1}(1-\beta x_\ell^2), \quad u_1 \; := \; \sum_{\ell\geqslant 1}\frac{x_\ell}{1-\beta x_\ell^2},\quad u_2 \; := \; \sum_{\ell\geqslant 1}\frac{x_\ell^2}{1-\beta x_\ell^2}, \quad u_3 \; := \; \sum_{\ell\geqslant 1 } \tanh^{-1}\left(x_\ell\sqrt{\beta}\right)/\sqrt{\beta}
\]
\vspace{.2cm}

\noindent To take the slant product of $F(x_1,x_2,\ldots)$ with $\Omega_g$, we use the following fact, which is provided by a slight restatement of \cite[Cor. to Lemma 3]{zagier} :

\vspace{.2cm}
\[
    \text{exp}\left(\Theta\kappa + \Xi \nu \right) \, / \; \Omega_k \; = \; \kappa^k\text{exp}\left(\nu^2\gamma/2\kappa\right).
\]
\vspace{.2cm}

\noindent Applying this to the expression (\ref{eq:F}), we obtain the expression

\vspace{.2cm}
\begin{equation}
    F(x_1,x_2,\ldots) \, / \, \Omega_k \; = \; 4^k u_1^k u_0^{g-1/2} \text{exp}\left( ((u_3-u_1)/\beta + u_2^2/u_1)\gamma^\ast  + (u_1-\beta u_2^2/u_1)\alpha \right).\label{eq:Fslant}
\end{equation}
\vspace{.2cm}

\noindent We are now in a position to use the following result of Zagier:\\

\vspace{.2cm}

\begin{lemma}[\cite{zagier} Prop. 3]\label{lemma:zagier}
    Let $f,h,u,w$ be power series in one variable, $h(0)u(0)\neq 0$. Then
    \begin{equation}
        \sum_{i\geqslant 1}  \left(f(\beta)h(\beta)^i e^{w(\beta)\gamma^\ast  + u(\beta) \alpha}\right) [N_i] \left(-\tfrac{1}{4} T\right)^{i-1} \; = \; \frac{\sqrt{\beta} f(\beta)M'(T)}{\sinh\left(\sqrt{\beta}(u(\beta) + \beta w(\beta))\right)}\biggr\rvert_{\beta = M(T)} \label{eq:zagier}
    \end{equation}
    where $M(T)$ is the power series defined by $M^{-1}(\beta) = \beta / u(\beta)h(\beta)$.\\
\end{lemma}

\vspace{.2cm}

\noindent We apply this in such a way as to avoid taking any functional inverses, i.e. such that $M(T)=T$. This is equivalent to choosing $h=1/u$. With this in mind, apply the lemma to (\ref{eq:Fslant}) with the following:

\vspace{.2cm}
\[
     w \; = \; (u_3-u_1)/\beta +u_2^2/u_1, \quad u \; = \; u_1 - \beta u_2^2/u_1, \quad f \; = \; 4^k u_1^k u_0^{g-1/2}  u^k, \quad h \; = \; 1/u
\]
\vspace{.2cm}

\noindent Now note that $u + \beta w = u_3$. Henceforth $\beta=T$. Thus the denominator in (\ref{eq:zagier}) is equal to

\vspace{.2cm}
\begin{equation}
     \sinh\left(\sum_{\ell \geqslant 1} \tanh^{-1}\left(x_\ell\sqrt{T}\right)\right) \; = \; \frac{1}{2} \prod_{\ell \geqslant 1} \left(\frac{1 + x_\ell\sqrt{T}}{1 - x_\ell\sqrt{T}}\right)^{1/2} - \frac{1}{2}\prod_{\ell \geqslant 1} \left( \frac{1 - x_\ell\sqrt{T}}{1 + x_\ell\sqrt{T}}\right)^{1/2}.\label{eq:denom}
\end{equation}
\vspace{.2cm}

\noindent After taking common denominators, with a bit of manipulation we see that (\ref{eq:denom}) is equal to 

\vspace{.2cm}
\[
     u_0(T)^{-1/2} \cdot \sum_{\substack{J \subset \{1,2,\ldots\}\\ |J| \text{ odd}}} \sqrt{T}^{|J|} \prod_{\ell\in J} x_\ell \; = \; R(T)^{-1/2}\sqrt{T} \cdot Q(T)
\]
\vspace{.2cm}

\noindent where $Q(T)$ is defined above, and we've observed that $R(T)=u_0(T)$. Thus the right hand side of (\ref{eq:zagier}) can be identified as the power series in $T$ with coefficients in $\Lambda$ given by

\vspace{.2cm}
\[
     4^k R^{g}(u_1^2- T u_2^2)^k/Q,
\]
\vspace{.2cm}

\noindent where $R=R(T)$, and so forth. The remaining step is to show that $R(u_1^2-Tu_2^2)=P$. Indeed, this implies that the above expression is equal to $4^k P^k R^{g-k}/Q$, from which the proposition is proven by taking the coefficient of $T^{k-1}$ on both sides of (\ref{eq:zagier}). To show $R(u_1^2-Tu_2^2)=P$, we first observe

\vspace{.2cm}
\[
     u_1 \pm \sqrt{T} u_2 \; = \;  \sum_{\ell\geqslant 1} \frac{x_\ell \pm \sqrt{T} x_\ell^2}{1- T x_\ell^2} \; = \; \sum_{\ell \geqslant 1} \frac{x_\ell}{1 \mp \sqrt{T} x_\ell}.
\]
\vspace{.2cm}

\noindent Now, set $u_0^\pm = \prod_{\ell \geqslant 1} (1 \pm \sqrt{T} x_\ell)$ so that $u_0=u_0^+u_0^-$. Then $u_0(u_1^2-Tu_2^2)$ is the product of $u_0^+(u_1-\sqrt{T} u_2)$ and $u_0^-(u_1+\sqrt{T} u_2)$, and treating these two factors separately leads to the expression

\vspace{.2cm}
\begin{equation*}
     u_0(u_1^2-Tu_2^2) \; = \; \left(\sum_{\ell \geqslant 1} x_\ell \prod_{\ell\neq k \geqslant 1} \left(1+\sqrt{T} x_k\right)\right)\left(\sum_{m \geqslant 1} x_m \prod_{m\neq n \geqslant 1} \left(1-\sqrt{T} x_n\right) \right).
\end{equation*}
\vspace{.2cm}

\noindent We can then multiply the two terms on the right to get the following, noting along the way that the coefficients in front of odd powers of $\sqrt{T}$ are zero, as expected:

\vspace{.2cm}
\begin{equation}
     u_0(u_1^2-Tu_2^2) \; = \; \sum_{\ell \geqslant 1} x_\ell^2 \prod_{\ell \neq k \geqslant 1} \left(1- T x_k^2\right) + 2 \sum_{\ell > m \geqslant 1} x_\ell x_m \left(1- T x_\ell x_m\right) \prod_{\substack{j \geqslant 1\\ j\neq m,\ell}} \left(1- T x_j^2\right).\label{eq:unaught}
\end{equation}
\vspace{.2cm}

\noindent Now we identify the monomials in the variables $x_\ell$ that appear in (\ref{eq:unaught}), in order the rewrite it in terms of monomial symmetric functions.  In the first sum on the right side of (\ref{eq:unaught}), the only monomials are of the form $x_{r_1}^2x_{r_2}^2\cdots x_{r_{i}}^2$. These are the monomials that appear in $m_\lambda$ for which $\lambda = (2^{i})$ is the partition with $i$ parts all equal to $2$. For each set of distinct indices $r_1,\ldots,r_{i}$, there are $i$ instances of the monomial $x_{r_1}^2x_{r_2}^2\cdots x_{r_{i}}^2$ in the sum under consideration, one for each time $\ell = r_j$ where $j=1,\ldots,i$. In other words, taking into account the signs, and keeping track of powers of $T$, we deduce

\vspace{.2cm}
\[
      \sum_{\ell \geqslant 1} x_\ell^2 \prod_{\ell \neq k \geqslant 1} \left(1- T x_k^2\right) \; = \; \sum_{i\geqslant 0} (-1)^{i}(i+1)m_{(2^{i+1})}T^{i}.
\]
\vspace{.2cm}

\noindent Now we consider the second sum on the right side of (\ref{eq:unaught}), in which we can count two kinds of monomials: those of the form $x_{\ell}x_{m}x_{r_1}^2\cdots x_{r_{i}}^2$ which belong to the partition $(2^{i} 1^2)$ and those of the form $x_{r_1}^2x_{r_2}^2\cdots x_{r_{i}}^2$ that belong, as before, to the partition $(2^i)$. The first kind are easy to count: apart from signs, there is exactly one. For the second kind, ignoring signs and powers of $T$, note that for any distinct indices $r_1,\ldots,r_i$ we get ${i \choose 2}$ many instances of the monomial $x_{r_1}^2x_{r_2}^2\cdots x_{r_{i}}^2$ after expanding, for the different possibilities of choosing which indices among the $r_j$ are $\ell$ and $m$. Thus

\vspace{.2cm}
\[
       2 \sum_{\ell > m \geqslant 1} x_\ell x_m \left(1- T x_\ell x_m\right) \prod_{\substack{j \geqslant 1\\ j\neq m,\ell}} \left(1- T x_j^2\right) \; = \; 2 \sum_{i\geqslant 0} (-1)^{i}\left(m_{(2^{i} 1^2)} + {i+1 \choose 2} m_{(2^{i+1})}\right)T^{i}
\]
\vspace{.2cm}

\noindent in which we interpret ${1\choose 2}=0$. Now, adding these two expressions involving monomial symmetric functions as in the right side of (\ref{eq:unaught}) easily leads to the expression that defines $P(T)$. This completes the proof of (\ref{eq:mg}).\\

The computation of (\ref{eq:ng}) is quite similar. First, $\cn(Z_g|_{N_g})$ is the restriction of (\ref{eq:F}) from $N_k\times J_k$ to the factor $N_k$, which simply sets $\Theta$ and $\xi$ to zero, evaluated against $[N_k]$. We then apply (\ref{eq:zagier}) to compute this rvaluation by setting $w=(u_3-u_1)/\beta$, $u=u_1$, $f=u_0^{g-1/2}u^k$ and $h=1/u$. The computation proceeds just as above, but is simpler. We obtain that $\cn(Z_g|_{N_g})$ is $1/2^{g-1}$ times the coefficient of $T^{k-1}$ of the expression $u_0^g u_1^k/Q$, and it is straightforward to identify $u_0u_1=U$.
\end{proof}

\vspace{.75cm}

\noindent We showed in the proof that $P^k R^{g-k} /Q$ in expression (\ref{eq:mg}) is equal to $4^ku_0^g(u_1^2-Tu_2^2)^k/Q$, and similarly $U^k R^{g-k} /Q$ in expression (\ref{eq:ng}) is equal to $u_0^gu_1^k/Q$. Note from the definitions that

\vspace{.3cm}
\[
    u_1(T) \; = \; \sum_{i\geqslant 0} p_{2i+1}T^i, \qquad \qquad u_2(T) \; = \; \sum_{i\geqslant 0}  p_{2i+2}T^i
\]
\vspace{.3cm}

\noindent in which $p_n$ is the $n^\text{th}$ power sum symmetric function.
\noindent We also mention that an expression for $P(T)$ in terms of elementary symmetric polynomial is given as follows:

\vspace{.2cm}
\[
    P(T) \; = \; \sum_{n\geqslant 0} (-1)^n\left(\widehat{e}_{n+1}^2 - 2\sum_{i=0}^{n-1}(-1)^i\widehat{e}_{n-i}\widehat{e}_{n+2+i}\right)T^n
\]
\vspace{.2cm}

\noindent where we have defined $\widehat{e}_k = ke_k$. We will not use this, and leave its verification to the reader.\\

\vspace{.3cm}

We also mention a generalization of (\ref{eq:mg}) which incorporates the classes $b_1^j$ into the pairings. The proof is a modification of the proof for Theorem \ref{thm:maincomputation}, and so we only briefly sketch it. The goal is to find a formula for a power series in $T$ whose coefficients are in $\Lambda[t]$, and such that the coefficient of $m_\lambda t^j T^{k-1}$ is the pairing of the monomial $z_{g,\lambda_1}\cdots z_{g,\lambda_n}B^j_1$ against $[M_k]$. To achieve this, since $p^\ast(B_1)=4\Theta$, within the exponential of (\ref{eq:F}) we add the term $4\Theta t$. We then proceed with the computation as before, and at the end, the extraction of the coefficient in front of $t^jT^{k-1}$ suitably normalized gives the pairings we want. Next, we observe that $B^j_1$ is the sum of $\prod_{i\in J}b_1^ib_1^{i+g}$ with $|J|=j$. The pairing for each term in this product with $z_{g,\lambda_1}\cdots z_{g,\lambda_n}$ against $[M_k]$ is the same, by $\text{Sp}(2g,\Z)$-invariance. This allows us to write the result in terms of the classes $b_1^j$. The result is:\\

\vspace{.4cm}

\begin{prop}\label{prop:maincomputation} For $g\geqslant k\geqslant 1$, $J\subset \{1,\ldots,g\}$, $j=|J|$, and $\lambda$ a partition with $|\lambda|=4k-3+j$:

\vspace{.1cm}

    \[
      z_{g,\lambda_1}\cdots z_{g,\lambda_m}\prod_{i\in J}b_1^ib_1^{i+g}[M_{k}]  \; = \;  \frac{(-1)^{k}}{2^{2k-1}}{j \choose k}\cdot \underset{ m_\lambda T^{k-1}}{\text{{\emph{Coeff}}}} \left[ R(T)^{g+k-2j}U(T)^{j}P(T)^{j-k}Q(T)^{-1}\right]
    \]
\end{prop}

\vspace{.65cm}

\noindent We can go further and try to incorporate the classes $b_2^j$. For this we may use the same method sketched above, but instead of only adding one formal variable $t$ to keep track of the powers of $B_1$ in the pairings, we add three variables to record separately the powers of $B_1$, $B_2$ and $B_{12}$. However, it is not clear that pairings between $z_{g,\lambda_1}\cdots z_{g,\lambda_n}$ and a general monomial in the classes $b_1^j$ and $b_2^j$ can be extracted from this data.

\vspace{.45cm}

\vspace{.15cm}

\subsection{Extracting pairings via specializations}\label{sec:spec}

Before proving Theorem \ref{thm:mainchern}, we digress and show how to recover formula (\ref{eq:a1}) from Theorem \ref{thm:maincomputation}. We then compute some other pairings in a similar way.\\

We use the well-known method of specializations in the theory of symmetric functions. There is a ring homomorphism, $\text{ex}:\Lambda \longrightarrow \Q$ from the ring of symmetric functions to the rationals, characterized, for example, by its evaluation on the monomial symmetric functions:

\vspace{.15cm}
\[
    \text{ex}\left(m_{(1^n)}\right) \; = \; 1/n!, \qquad  \text{ex}\left(m_{\lambda}\right) \; = \; 0 \text{ if } \lambda \neq (1^n) \text{ for some } n
\]
\vspace{.15cm}

\noindent In particular, if $f\in \Lambda$ is a homogeneous symmetric function of degree $n$, then we have

\vspace{.15cm}
\[
    \text{ex}\left(f\right) \; = \; \frac{1}{n!}\cdot \underset{{x_1x_2\cdots x_n}}{\text{Coeff}}\left[f\right]
\]
\vspace{.15cm}

\noindent This homomorphism is a version of what's often called the {\emph{exponential specialization}} for symmetric functions. In general, a {\emph{specialization}} is just a homomorphism from $\Lambda$ to another ring. The homomorphism $\text{ex}$ just defined extends in the obvious way to a homomorphism $\text{ex}:\Lambda[[T]]\longrightarrow \Q[[T]]$. We can then directly apply this homomorphism to our previously defined power series:

\vspace{.1cm}
\[
    \text{ex}\left(R(T)\right) \; = \; 1, \qquad \text{ex}\left(P(T)\right) \; = \; 1, \qquad \text{ex}\left(Q(T)\right) \; = \; \sinh\sqrt{T}/\sqrt{T}.
\]
\vspace{.1cm}

\noindent We can now see that applying $\text{ex}$ to the computation (\ref{eq:mg}) of Theorem \ref{thm:maincomputation} with $g=k$ yields

\vspace{.3cm}
\[
    z_{1}^{4g-3}[M_g] \; = \; (4g-3)!\cdot \text{ex}\left(\cn\left( Z_g\right)\right) \; = \; (4g-3)!\cdot\frac{(-1)^g}{2^{2g-1}} \underset{T^{g-1}}{\text{Coeff}}\left[\frac{\sqrt{T}}{\sinh\sqrt{T}}\right].
\]
\vspace{.3cm}

\noindent At this point we recall an identity for the Bernoulli numbers, which may as well be taken as a convenient definition of $B_n$ for our purposes, which holds for even indices $n$:

\vspace{.15cm}
\[
    \underset{T^{g-1}}{\text{Coeff}}\left[\frac{\sqrt{T}}{\sinh \sqrt{T}}\right] \; = \; -\frac{(2^{2g-2}-2)B_{2g-2}}{{(2g-2)}!}.
\]
\vspace{.15cm}

\noindent Finally, the identity $a_1=-2z_{1}$ recovers formula (\ref{eq:a1}), the expression for the pairing of the top degree power of the class $a_1$. We can similarly recover the formula (\ref{eq:alphabeta}) with $i=3g-3$ for the pairing $\alpha^{3g-3}[N_g]$ upon observing $\text{ex}(U(T))=1$.\\

We can generalize the discussion and perform a similar extraction to obtain a formula for pairings of the form $a_{1}^{4g-3-k}z_{k}[M_g]$. For this, we consider the specialization $\overline{\text{ex}}:\Lambda\longrightarrow \Q[x]$ characterized by sending a homogeneous symmetric function $f$ of degree $n$ to the following:

\vspace{.15cm}
\[
    \overline{\text{ex}}\left(f\right) \; = \; \sum_{k \geqslant 0} \frac{x^{k}}{(n-k)!} \cdot \underset{x_1^k x_2 x_3 \cdots x_{n-k+1} }{\text{Coeff}}\left[f\right].
\]
\vspace{.15cm}

\noindent This extends in the obvious way to a homomorphism $\overline{\text{ex}}:\Lambda[[T]]\longrightarrow \Q[x][[T]]$, and is equal to the\\

\noindent above $\text{ex}$ if we set $x=0$. From the definition note that for $n>0$ we have

\vspace{.15cm}
\[
    \overline{\text{ex}}\left(m_{(1^n)}\right) \; = \; x/(n-1)! + 1/n!, \qquad  \overline{\text{ex}}\left(m_{(1^{n-k}k)}\right) \; = \; x^k/(n-k)! \;\;\text{  for }\; k >1.
\]
\vspace{.15cm}

\noindent We can apply this homomorphism to our power series just as before, and we get:

\vspace{.15cm}
\[
    \overline{\text{ex}}\left(R(T)\right) \; = \; 1, \qquad \overline{\text{ex}}\left(P(T)\right) \; = \; (1+x)^2 - x^2 T, \qquad \overline{\text{ex}}\left(Q(T)\right) \; = \; x\cosh\sqrt{T} + \sinh\sqrt{T}/\sqrt{T}.
\]
\vspace{.15cm}

\noindent Then applying the homomorphism $\overline{\text{ex}}$ to the formula in (\ref{eq:mg}) with $g=k$ yields the following:

\vspace{.3cm}
\[
    z_{1}^{4g-3-i}z_{i}[M_g] \; = \; (4g-3-i)!\cdot \frac{(-1)^g}{2^{2g-1}}\cdot \underset{x^i T^{g-1}}{\text{Coeff}}\left[\frac{(1+2x+x^2-x^2 T)^g}{x\cosh\sqrt{T} + \sinh\sqrt{T}/\sqrt{T}}\right]
\]
\vspace{.3cm}

\noindent We also see in this situation that a recursion property holds for lower genus moduli spaces: for $1\leqslant k\leqslant g$, we have $z_{g,1}^{4k-3-i}z_{g,i}[M_k] = z_{1}^{4k-3-i}z_{i}[M_k]$. Such a recursion always holds for any pairings obtained from a specialization that sends $R(T)$ to $1$. Similarly, noting that $\overline{\text{ex}}$ applied to $U(T)$ yields $1+x-x^2T$, we obtain the following formula for $N_g$ by applying $\overline{\text{ex}}$ to (\ref{eq:ng}) with $g=k$:

\vspace{.3cm}
\[
    \xi_{1}^{3g-3-i}\xi_{i}[N_g] \; = \;(3g-3-i)!\cdot\frac{1}{2^{g-1}}\cdot \underset{x^i T^{g-1}}{\text{Coeff}}\left[\frac{(1+x-x^2 T)^g}{x\cosh\sqrt{T} + \sinh\sqrt{T}/\sqrt{T}}\right]
\]
\vspace{.3cm}

\noindent Again, using $\xi_1=-\alpha/2$ and relation (\ref{eq:defxi}), this determines the pairings $\alpha^{3g-3-k}\delta_k[N_g]$. A similar recursive property as was mentioned above also holds for these pairings.\\

In a different direction, we can define specializations by setting some of the variables $x_1,x_2,\ldots$ in the definition of the symmetric functions equal to zero. For example, setting $x_1=x$, $x_2=y$, and $x_\ell =0$ for $\ell \geqslant 3$, we obtain a specialization $\text{ev}_{2}:\Lambda[[T]]\longrightarrow \Z[x,y][[T]]$ which acts as follows:

\vspace{.3cm}
\[
    \text{ev}_{2}\left(U(T)\right) \; = \; x(1-y^2T) + y(1-x^2 T),  \qquad \text{ev}_{2}\left(Q(T)\right) \; = \; x+y.
\]
\vspace{.3cm}

\noindent We then obtain a formula for the pairings $\xi_i \xi_j [N_g]$ with $i+j=3g-3$ by applying the specialization $\text{ev}_{2}$ to equation (\ref{eq:ng}) with $g=k$, after some elementary manipulations:

\vspace{.3cm}
\begin{equation}
    \sum_{i+j=3g-3}\xi_i \xi_j [N_g]\cdot x^iy^i \;  = \; g\cdot \frac{(-1)^{g-1}}{2^{g-1}}\cdot (xy^2+yx^2)^{g-1}\label{eq:xipair}
\end{equation}
\vspace{.3cm}

\noindent Recall from Prop. \ref{prop:xigens} that we may view $\xi_i$ as generators of the cohomology of $N_g$ with $\Z_p$ coefficients when $p$ is odd. Suppose $p=g-1$ is an odd prime. In (\ref{eq:xipair}), the constants in front of the right-hand expression are invertible mod $p$ (interpreting $1/2^{g-1}$ as the inverse of $2^{g-1}$ mod $p$), and we conclude that the only pairing of the form $\xi_i\xi_j [N_g]$ that is nonzero mod $p$ is given by $\xi_{g-1}\xi_{2g-2}[N_g]$.\\

\vspace{.65cm}

\subsection{Chern numbers for the tangent bundle}\label{sec:cherns}

In this subsection we prove Theorem \ref{thm:mainchern}. The proof is similar to that of Theorem \ref{thm:maincomputation}, so we only indicate where it differs. We write $TN_g$ for the tangent bundle of $N_g$, viewed as complex vector bundle.\\

\vspace{.2cm}

\begin{proof}[Proof of Thm. \ref{thm:mainchern}] According to Zagier \cite[eq. (27)]{zagier}, the Chern class polynomial of $N_g$ is:

\vspace{.3cm}
\[
    c(TN_g)_x \; = \; (1-\beta x^2)^{g-1}\text{exp}\left(\frac{2\alpha x}{1-\beta x^2} + 2\left(\frac{\tanh^{-1}x\sqrt{\beta}}{\beta\sqrt{\beta}} - \frac{x}{\beta (1-\beta x^2)} \right)\gamma^\ast \right)
\]
\vspace{.3cm}

\noindent where as before $\gamma^\ast = 2\gamma +\alpha\beta$. Note the relation $c(TN_g)_x = (1-\beta x^2)^{-g}c(Z_g|_{N_g})_{-2x}^2$. Proceeding as in the proof of Theorem \ref{thm:maincomputation}, the Chern number polynomial $\cn(TN_g)$ is given by $F_0(x_1,x_2,\ldots)[N_g]$ where $F_0$ is the product of the Chern class polynomials $c(TN_g)_{x_\ell}$ for $\ell\geqslant 0$:

\vspace{.3cm}
\begin{equation*}
    F_0(x_1,x_2,\ldots) \; = \; u_0^{g-1}\text{exp}\left(2(u_3-u_1)\gamma^\ast/\beta  + 2u_1\alpha\right).
\end{equation*}
\vspace{.3cm}

\noindent Here the expressions for $u_0,u_1,u_3$ are defined as before. Now we apply Lemma \ref{lemma:zagier} as was done previously, but with $w = 2(u_3-u_1)/\beta$,  $u = 2u_1$,  $f =  u_0^{g-1}  u^g$ and $h=1/u$. From this we obtain

\vspace{.3cm}
\begin{equation}
    \cn(TN_g) \; = \;  2^{g}(-4)^{g-1}\frac{\sqrt{T}\cdot u_0(T)^{g-1}u_1(T)^g}{\sinh\left(2\sum_{\ell \geqslant 1} \tanh^{-1}\left(x_\ell\sqrt{T}\right)\right)}\label{eq:cn:chern}
\end{equation}
\vspace{.3cm}

\noindent The denominator here is computed as in (\ref{eq:denom}), but now the right side of (\ref{eq:denom}) loses the fractional $1/2$ exponents due to the presence of the $2$ in (\ref{eq:cn:chern}). After a short manipulation we instead find

\vspace{.3cm}
\begin{equation*}
    \sinh\left(2\sum_{\ell \geqslant 1} \tanh^{-1}\left(x_\ell\sqrt{T}\right)\right) \; = \; \frac{1}{2} \prod_{\ell \geqslant 1} \left(\frac{1 + x_\ell\sqrt{T}}{1 - x_\ell\sqrt{T}}\right) - \frac{1}{2}\prod_{\ell \geqslant 1} \left( \frac{1 - x_\ell\sqrt{T}}{1 + x_\ell\sqrt{T}}\right) \; = \; {\frac{2\cdot Q(T)E(T)\sqrt{T}}{u_0(T)}}
\end{equation*}
\vspace{.3cm}

\noindent where $E(T)$ is defined in the introduction, and is readily identified with $\sum T^{|J|/2} \prod_{\ell\in J} x_\ell$, the sum being over finite subsets $J \subset \{1,2,\ldots\}$ of even cardinality. Finally, recalling that $u_0 u_1 = U$, we obtain from (\ref{eq:cn:chern}) the formula given in Theorem \ref{thm:mainchern}.
\end{proof}

\vspace{.7cm}

The specializations of Section \ref{sec:spec} can, of course, also be applied to this situation. As an example, since $\text{ex}(E(T))=\cosh \sqrt{T}$, applying the specialization $\text{ex}$ to Theorem \ref{thm:mainchern} yields

\vspace{.2cm}

\[
    c_1(TN_g)^{3g-3}[N_g] \; = \; (3g-3)!(-2)^{3g-3}\; \underset{T^{g-1}}{\text{Coeff}}\left[\frac{\sqrt{T}}{\sinh\sqrt{T}\cosh\sqrt{T}}\right].
\]

\vspace{.35cm}

\noindent Then, using that $c_1(TN_g)=2\alpha$, and the hyperbolic-trig identity $\sinh(2x)=2\sinh(x)\cosh(x)$, we once again recover the formula for $\alpha^{3g-3}[N_g]$ given in (\ref{eq:alphabeta}).

\newpage

\vspace{.75cm}

\subsection{Skew Schur functions}\label{sec:skew}

This section is mostly expositional, and serves to explain how the power series $1/Q(T)$ and $1/E(T)$ are generating functions for certain skew Schur symmetric functions as mentioned in the introduction. In particular, we explain (\ref{eq:qinv}). This interpretation was pointed out to the authors by Ira Gessel, and appears as a particular example in Section 11 of \cite{gesselviennot}. The reader may consult I.5 of \cite{mac} for more background on skew Schur functions.\\

We begin by defining skew Schur functions. To begin, for any partition $\lambda$, the Schur symmetric function $s_\lambda$ associated to $\lambda$ is defined as the determinant $\det(h_{\lambda_i+j-i})_{1\leqslant i,j\leqslant k}$ in which $h_{\lambda}$ is the complete monomial symmetric function (see appendix \ref{appendix}), and $k$ is the length of the partition $\lambda$. One of the {\emph{Jacobi-Trudi}} identities says that $s_\lambda$ is also equal to $\det(e_{\lambda'_i+j-i})_{1\leqslant i,j\leqslant k}$ where $\lambda'$ is the partition conjugate to $\lambda$. More generally, the {\emph{skew}} Schur symmetric function $s_{\lambda/\mu}$ is equal to $\det(h_{\lambda_i-\mu_i+j-i})_{1\leqslant i,j\leqslant k}$. By a Jacobi-Trudi identity, we have the following identity:

\vspace{.18cm}
\begin{equation}
    s_{\lambda/\mu} \; = \; \det(e_{\lambda'_i-\mu'_j+j-i})_{1\leqslant i,j\leqslant k}\label{eq:skews}
\end{equation}
\vspace{.1cm}

\noindent In this situation, $\mu$ is always a subpartition of $\lambda$, and the pair of data $(\lambda,\mu)$ is often called a skew partition, and written $\lambda/\mu$.\\

We turn to some general remarks on generating functions and determinants which are standard in enumerative combinatorics, see e.g. \cite{stanley}. Suppose that  $a_i$ with $i\geqslant 0$ are a list of elements in some commutative ring. Then the reciprocal of the generating function $\sum_{i\geqslant 0} a_i T^i$ has coefficients in terms of some determinants formed from the $a_i$ up to some powers of $a_0$:

\vspace{.3cm}
\[
    \sum_{n \geqslant 0} \frac{1}{a_0^{n+1}}\det \left((-1)^{j-i+1}a_{j-i+1}\right)_{1\leqslant i,j\leqslant n} \, T^n \; = \; \frac{1}{\sum_{i\geqslant 0}  a_i T^i}
\]
\vspace{.3cm}

\noindent As written, we are assuming the element $a_0$ is invertible. More generally, as long as $a_0$ is not a zero-divisor, then the coefficient of $T^n$ on the left hand side after multiplying by $a_0^{n+1}$ is a well-defined element of the ring with which we started. Now we set $a_i = e_{2i+1}$ so that the right hand side is equal to $1/Q(T)$. In this application, the commutative ring is the ring of symmetric polynomials, and $a_0=e_1$. Then, defining $r_n$ to be $e_1^{n+1}$ times the coefficient of $T^n$ in $1/Q(T)$ we obtain that $r_n = \det((-1)^{j-i+1}e_{2j-2i+3})_{1\leqslant i,j\leqslant n}$. Upon observing that $r_n$ is a homogeneous symmetric polynomial of degree $3n$, expanding the determinant allows us to factor out the sign, and we obtain:

\vspace{.3cm}
\begin{equation}
    r_n \; = \; (-1)^n\det\left(e_{2j-2i+3}\right)_{1\leqslant i,j\leqslant n}.\label{eq:rn0}
\end{equation}
\vspace{.3cm}

\noindent Now, $\lambda(n,3)' = (n+2,n+1,\ldots,4,3)$ and $\lambda(n,0)'=(n-1,n-2,\ldots,2,1)$, where $\lambda(n,m)$ is defined in the introduction. It follows from (\ref{eq:skews}) and (\ref{eq:rn0}) that $r_n$ is equal to $(-1)^ns_{\lambda(n,3)/\lambda(n,0)}$. This establishes formula (\ref{eq:qinv}) for $1/Q(T)$, and $1/E(T)=\sum s_{\lambda(n,2)/\lambda(n,0)}(-T)^n$ is similarly obtained.\\

\begin{figure}[t]
\centering
\includegraphics[scale=.95]{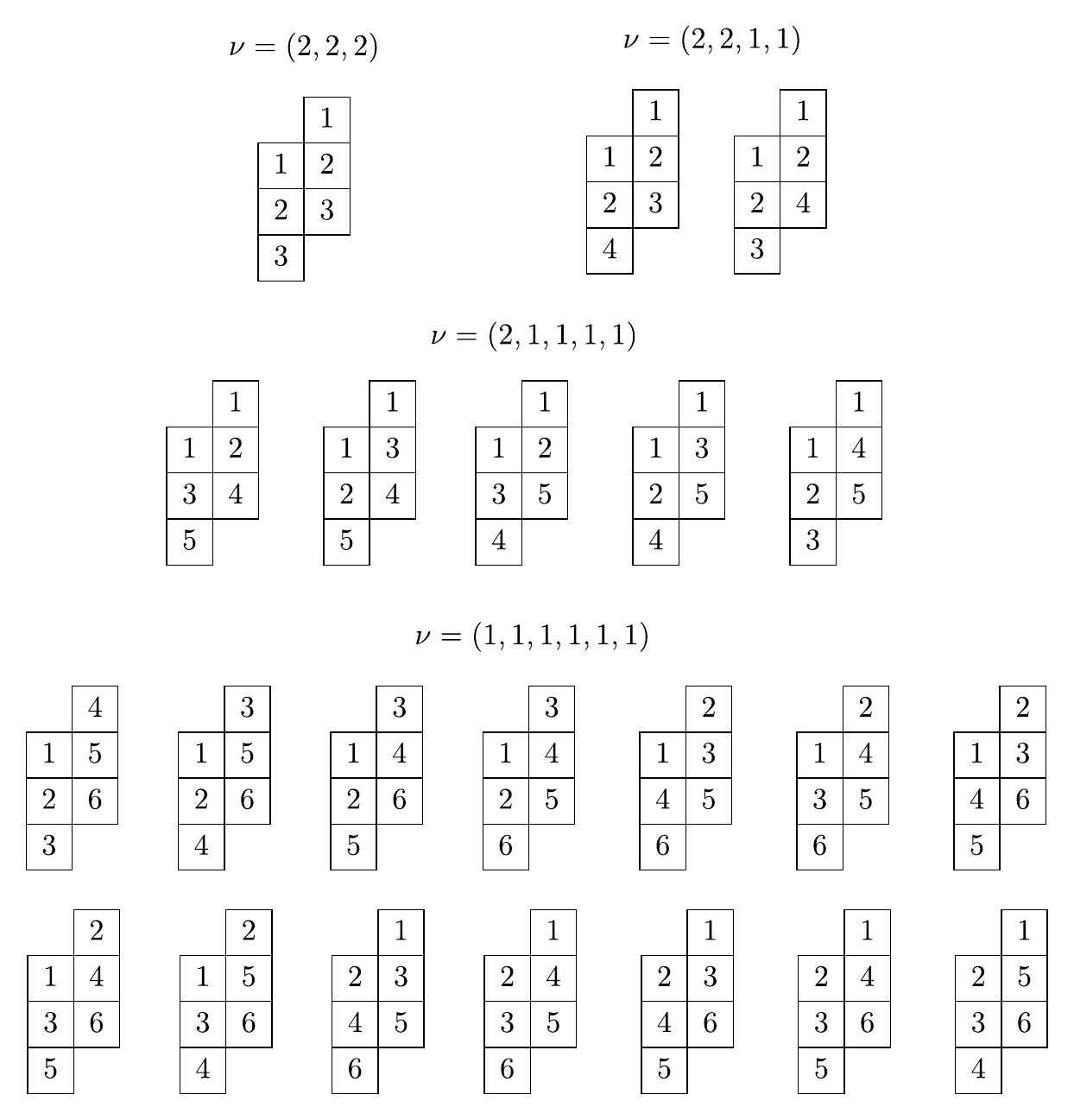}
\vspace{.3cm}
\caption{{\small{Here we list the SSYT of shape $\lambda_2/\mu_2$ with type $\nu$ for each partition $\nu$. Using equation (\ref{eq:ssyt}) we conclude $s_{\lambda_2/\mu_2} = m_{(2,2,2)} + 2m_{(2,2,1,1)} + 5m_{(2,1,1,1,1)} + 14m_{(1,1,1,1,1,1)} $}.
\vspace{.4cm}}}\label{fig:ssyt}
\end{figure}

The skew Schur function $s_{\lambda/\mu}$ admits the following combinatorial interpretation. Let $\lambda/\mu$ be any skew partition. Define a {\emph{semi-standard (skew) Young tableau}} (SSYT) {\emph{of shape}} $\lambda/\mu$ to be a filling of $\lambda/\mu$ with positive integers that are non-decreasing from left to right in each row and strictly increasing from top to bottom in each column. If a SSYT of shape $\lambda/\mu$ has $\alpha_i$ instances of $i$ for each positive integer $i$, we say that the {\emph{type}} of the SSYT is the composition $\alpha=({\nu_1},{\nu_2},\ldots)$. Then we have the following identity, in which the sum is over all partitions $\nu$, thought of in this context as compositions of non-increasing non-negative integers:

\vspace{.3cm}
\begin{equation}
    s_{\lambda/\mu} \; = \; \sum_{\nu} K_{\lambda/\mu,\nu}m_{\nu}, \qquad  K_{\lambda/\mu,\nu} \; = \;\#\{\text{SSYT of shape }\lambda/\mu \text{ and type }\nu\} \label{eq:ssyt}
\end{equation}
\vspace{.3cm}

\noindent The numbers $K_{\lambda/\mu,\nu}$ are called the (skew) {\emph{Kostka numbers}}. The example for $s_{\lambda_2/\mu_2}$ is spelled out in Figure \ref{fig:ssyt}. These numbers can become large quite fast: the coefficient in front of $m_{(1^9)}$ within $s_{\lambda_3/\mu_3}$ is equal to $744$, and in front of $m_{(1^72)}$ is equal to 323.\\

The relationship we have established between the integral pairings on the moduli spaces and these skew Schur functions will not be exploited in this paper, although some of the arguments below may have combinatorial interpretations.

\vspace{.65cm}

\section{Mod two nilpotency}\label{sec:nilp}

In this section we prove Theorem \ref{thm:nilp}. We first consider the degree 2 class $\alpha$ in the cohomology of $N_g$, and later handle the corresponding class $a_1$ in the cohomology of $M_g$. We begin by showing that $\alpha^g$ is zero mod 2. In fact, we have more generally the following:

\vspace{.45cm}

\begin{prop}\label{prop:2divalpha}
    For $n\geqslant g-1$, the element $\alpha^{n}$ is divisible by $2^{n-g+1}$.
\end{prop}

\vspace{.4cm}

\begin{proof}
From Section \ref{sec:ng}, we gather that the residue classes of $\alpha,\delta_{i},\psi_j$ for $2\leqslant i \leqslant 2g-1$ and $1\leqslant j\leqslant 2g$ generate the mod $2^m$ cohomology ring of $N_g$ for any $m\geqslant 1$, and in particular $m=n-g+1$. It suffices then to show that for every partition $\lambda$, subset $J\subset\{1,\ldots,2g\}$, and $\ell\geqslant 0$ we have

\vspace{.3cm}

\[
    \alpha^{n+\ell}\delta_{\lambda_1}\cdots\delta_{\lambda_k}\prod_{j\in J}\psi_j [N_g] \; \equiv \; 0 \mod 2^{n-g+1}.
\]

\vspace{.3cm}

\noindent Now recall that $\alpha = -2\xi_1$, and from (\ref{eq:defxi}) that each $\delta_i$ is an integral combination of terms $\xi^b_1\xi_j$. The above pairing is then an integral combination of pairings of the form

\vspace{.3cm}

\begin{equation}
    2^{n+\ell}\xi_{\nu_1}\cdots\xi_{\nu_r}\prod_{j\in J}\psi_J [N_g]\label{eq:pring}
\end{equation}

\vspace{.3cm}

\noindent Now either $J$ is not invariant under the involution $j \mapsto j+g$ (mod 2g), in which case (\ref{eq:pring}) is zero, or else (\ref{eq:pring}) is the coefficient of $m_\nu$ within $2^{n+\ell}\cn(Z_g|_{N_k})$ where $k = g - |J|/2$. It is apparent from Theorem \ref{thm:maincomputation} that $2^n\cn(Z_g|_{N_k})$ has coefficients divisible by $2^{n-g+1}$, since the power series inside the brackets of (\ref{eq:ng}) has coefficients symmetric functions with integer coefficients.
\end{proof}

\vspace{.5cm}

\noindent The nilpotency degree of $\alpha$ mod 2 is then computed by the following, which implies that $\alpha^{g-1}$ is nonzero in the cohomology ring $H^\ast(N_g;\Z_2)$:\\

\vspace{.4cm}

\begin{lemma}\label{lemma:xi}
    The parity of the integer $2^{g-1}\xi_1^{3g-3-j}\xi_{j}[N_g]$ is determined as follows:
    
    \[
        2^{g-1}\xi_1^{3g-3-j}\xi_{j}[N_g] \; \equiv \; 1\;\; {\emph{\text{(mod 2)}}} \quad \Longleftrightarrow \quad \begin{cases} j\in \{g-1,\; g-2\} & g \text{ even}\\ j\in \{g,\;g-1\} & g \text{ odd} \end{cases}
    \]
\end{lemma}

\vspace{.4cm}

\begin{proof}
By Theorem \ref{thm:maincomp}, the term $2^{g-1}\xi_1^{3g-3-j}\xi_{j}[N_g]$ is equal to the coefficient of $m_{\lambda}T^{g-1}$ within $U(T)^g/Q(T)$ where $\lambda=(j,1^{3g-3-j})$. We use this to reformulate the claim of the lemma as follows. Let $I\subset\Lambda$ be the ideal generated by $2\in\Z$ and the monomial symmetric functions $m_{\lambda}$ with $\lambda$ having at least two parts greater than 1, i.e. $\lambda_1,\lambda_2>1$. Here, as before, $\Lambda$ is the ring of symmetric functions with integer coefficients. Then the lemma is equivalent to the congruence

\vspace{.0cm}

\begin{equation}
    \underset{T^{g-1}}{\text{Coeff}} \Big[\,U(T)^g/Q(T)\,\Big] \; \equiv \;\;g\cdot m_{(g,\, 1^{2g-3})} + m_{(g-1,\, 1^{2g-2})} + (g-1)\cdot m_{(g-2,\, 1^{2g-1})} \qquad (\text{mod  } I)\label{eq:coeff1}
\end{equation}

\vspace{.3cm} 

\noindent Let $J\subset I$ be the ideal generated only by $m_{\lambda}$ with $\lambda$ having at least two parts greater than 1, omitting $2\in \Z$. The following relations are easily verified in the quotient ring $\Lambda/J$, in which $r,s>1$:

\vspace{-.1cm}

\begin{align*}
    m_{(1^p)}m_{(1^q)} & \; \equiv \; {p+q\choose p}m_{(1^{p+q})}  \;\;\;\;+\; {p+q-2\choose p-1}m_{(2,\,1^{p+q-2})} & (\text{mod  } J)\\
    {}\\
     m_{(r,\,1^p)}m_{(1^q)} & \; \equiv \; {p+q\choose p}m_{(r,\,1^{p+q})}  \;+\; {p+q-1\choose p}m_{(r+1,\,1^{p+q-1})} & (\text{mod  } J)\\
     {}\\
      m_{(r,\,1^p)}m_{(s,\, 1^q)} & \; \equiv \; {p+q\choose p}m_{(r+s,\,1^{p+q})} & (\text{mod  } J)
\end{align*}

\vspace{.3cm} 

\noindent These follow by simply expanding the monomial symmetric functions as the sums of monomials that define them, and multiplying. Along the same lines, we leave the following to the reader:

\vspace{.3cm}

\begin{equation}
    m_{(1^p)}^n \; \equiv \; \sum_{k=1}^{n} {n\choose k}(p n-k)!p^kp !^{-n} m_{(k, \, 1^{p n-k})} \qquad\qquad ({\text{mod  }} J)\label{eq:lucas}
\end{equation}

\vspace{.3cm}

\noindent We list some cases for which (\ref{eq:lucas}) vanishes modulo $I$. First, a special case of an elementary result, often called Lucas's Theorem, says that a multinomial coefficient $(\alpha_1+\cdots + \alpha_k)!/\alpha_1!\cdots \alpha_n!$ is even if and only if, in the binary expansions of the $\alpha_i$, there is some position (i.e. digit location) for which two distinct $\alpha_i$ have digit equal to $1$. Next, $(pn-k)! p^k p!^{-n}$ is equal to the multinomial coefficient in which $\alpha_1=\cdots=\alpha_{k}=p-1$ and $\alpha_{k+1}=\cdots=\alpha_n=p$. If $n\geqslant 3$, then either $p$ or $p-1$ appears at least twice, so by Lucas's Theorem this number is even. Thus $m^n_{(1^p)} \equiv 0$ (mod $I$) if $n\geqslant 3$. If $n=2$, the $k=1$ term in (\ref{eq:lucas}) drops out because of $\binom{n}{k}$, so we need only consider $n=2=k$. This case has the term $\binom{2p-2}{p-1}$, which is even, unless $p \leq 2$. Thus $m^n_{(1^p)} \equiv 0$ (mod $I$) if $p\geqslant 3$ and $n=2$.\\

As in Section \ref{sec:skew}, define $r_n\in\Lambda$ to be  $e_1^{n+1}$ times the coefficient of $T^n$ of $1/Q(T)$. The general formula for the reciprocal of a power series applied to $1/Q(T)$ yields

\vspace{.3cm}

\[
    r_n \; = \; \sum_{\substack{\alpha_1,\ldots,\alpha_k \geqslant 0,\; k\geqslant 1\\ \alpha_1 + 2\alpha_2 +\cdots + k\alpha_k = n}}(-1)^{\alpha_1+\cdots+\alpha_k}{\alpha_1+\cdots+\alpha_k \choose \alpha_1,\ldots,\alpha_k}e_1^{n-\sum \alpha_i}e_3^{\alpha_1}\cdots e_{2k+1}^{\alpha_k}
\]

\vspace{.3cm}

\noindent Now, recalling that $e_p = m_{(1^p)}$ and taking our above remarks regarding $e_p^n$ into consideration, we see that only terms with all $\alpha_i \leqslant 2$ can contribute odd coefficients. Further, by Lucas's Theorem, if all $\alpha_i\leqslant2$, then the multinomial coefficient appearing is even unless  $\{\alpha_1,\ldots,\alpha_k\}$ has either (i) some $i$ such that $\alpha_i=1$ and $\alpha_j = 0$ for $j\neq i$, (ii) some $i$ such that $\alpha_i=2$ and $\alpha_j = 0$ for $j\neq i$, or (iii) some $i,j$ such that $\alpha_i=2$, $\alpha_j = 1$, and $\alpha_\ell=0$ for $\ell\neq i,j$. Only case (i) actually contributes something nonzero mod 2, since we remarked above that $m_{(1^p)}^2$ is even when $p\geqslant 3$. We conclude:

\vspace{.3cm}

\[
    r_n \; \equiv \; e_1^{n-1}e_{2n+1} \qquad ({\text{mod  }} I)
\]

\vspace{.3cm}

\noindent From the relations above we easily compute $e_1^{n-1}$ modulo $I$. It is congruent to $m_{(n-2,\, 1)} + m_{(n-1)}$ if $n$ is even, and simply $m_{(n-1)}$ if $n$ is odd. From this we obtain

\vspace{.3cm}

\begin{equation}
    r_n \; \equiv \;  m_{(n,\, 1^{2n})} + n\cdot m_{(n-1,\, 1^{2n+1})}  \qquad ({\text{mod  }} I)\label{eq:rn}
\end{equation}

\vspace{.3cm}

\noindent Note that $r_n$ can be replaced here with the skew Schur function $s_{\lambda_n/\mu_n}$, and for $n=2$ the congruence (\ref{eq:rn}) is apparent from Figure \ref{fig:ssyt}. Next, since $U(T) \equiv m_{(1)} - m_{(2,1)}T$ (mod $J$), the coefficient of $T^i$ within $U(T)^g$ is congruent to $(-1)^i{g\choose i}e_1^{g-i}m_{(2,1)}^i$ modulo $J$. We then gather the following:

\vspace{.3cm}

\begin{equation}
    \underset{T^{g-1}}{\text{Coeff}} \Big[\,U(T)^g/Q(T)\,\Big] \; = \; \sum_{i=0}^{g-1} \underset{T^{i}}{\text{Coeff}} \Big[\,U(T)^g\,\Big] e_1^{i-g}r_{g-i-1} \; \equiv \; \sum_{i=0}^{g-1} {g \choose i}  m_{(2,1)}^{i} r_{g-i-1}\qquad (\text{mod } I)\label{eq:longcoeff}
\end{equation}

\vspace{.3cm}

\noindent A quick check shows that $m_{(2,1)}^i\equiv 0$ (mod $I$) for $i\geqslant 2$, so the only terms contributing are at $i=0,1$. Thus the sum is congruent to $r_{g-1} + g\cdot m_{(2,\, 1)}r_{g-2}$, which with (\ref{eq:rn}) computes (\ref{eq:coeff1}).
\end{proof}

\vspace{.5cm}

\begin{corollary}\label{cor:1pairng}
    The pairing $\alpha^{g-1}\delta_{2g-2}[N_g]$ is odd.
\end{corollary}

\vspace{.54cm}

\begin{proof}
Using formula (\ref{eq:defxi}) and $\alpha=-2\xi_1$, we extract the relation

\vspace{.3cm}

\[
    \alpha^{g-1}\delta_{2g-2} \; = \; \sum_{i=1}^{2g-1}(-1)^{i+g} i \cdot 2^{g-1}\xi_1^{g-2+i}\xi_{2g-1-i}
\]

\vspace{.3cm}

\noindent Noting the coefficient $i$, we see that exactly one of the terms from Lemma \ref{lemma:xi} contributes an odd number once we pair with $[N_g]$.
\end{proof}

\vspace{.64cm} 

\noindent This corollary, together with Prop. \ref{prop:2divalpha}, proves that the nilpotency degree of $\alpha$ as viewed in $H^\ast(N_g;\Z_2)$ is equal to $g$. For the second part of Theorem \ref{thm:nilp}, regarding the nilpotency degree of $a_1$ in the ring $H^\ast(M_g;\Z_2)$, we establish an analogue of Corollary \ref{cor:1pairng}. First:\\

\vspace{.5cm}

\begin{lemma}
    The integer $2^{2g-1}z_1^{j}z_{4g-3-j}[M_g]$ is odd if and only if $j\in \{2g-1,2g-2\}$.
\end{lemma}

\vspace{.54cm}

\begin{proof}
We sketch the proof, which is similar to that of Lemma \ref{lemma:xi}. By (\ref{eq:mg}) of Theorem \ref{thm:maincomputation}, this integer is the coefficient of $m_{\lambda}T^{g-1}$ within $P(T)^g/Q(T)$ where $\lambda = (4g-3-j, 1^j)$. Since $P(T)$ is congruent modulo $J$ to $e_1^2 - 2m_{(2,1,1)}T$, the only term in $P(T)^g$ relevant to $\Lambda/I$ is the constant term $e_1^{2g}$. Then the coefficient of $T^{g-1}$ within $P(T)^g/Q(T)$, computed just as in (\ref{eq:longcoeff}), is congruent mod $I$ to $e_1^{g}r_{g-1}$. From (\ref{eq:rn}) this is then $m_{(2g-1, 1^{2g-2})} + m_{(2g-2,1^{2g-1})}$ mod $I$, proving the claim.
\end{proof}

\vspace{.5cm}

\noindent The same argument as in the proof of Corollary \ref{cor:1pairng} then yields the following:

\vspace{.75cm}

\begin{corollary}\label{cor:1pairmg}
    The pairing $a_1^{2g-1}d_{2g-2}[M_g]$ is odd.
\end{corollary}

\vspace{.86cm}

\noindent To complete the proof of Theorem \ref{thm:nilp} it remains to show that $a_1^{2g}$ is zero mod 2. For this we will prove an analogue of Proposition \ref{prop:2divalpha}. To this end, we first establish a few lemmas which take the content of Section \ref{sec:intersection} a bit further.\\

We begin by sketching the geometric meaning of Thaddeus's genus recursive formula (\ref{eq:psi}). Recall that $N_g$ may be viewed as the space of conjugacy classes of $2g$-tuples $(A_{i})_{i=1}^{2g}$ in the $2g$-fold product of $SU(2)$ such that the product of the commutators $[A_i,A_{i+g}]$ for $1\leqslant i\leqslant g$ is equal to $-1$. For $I\subset\{1,\ldots,2g\}$, let the submanifold $N_I \subset N_g$ consist of conjugacy classes such that $A_i = 1$ if $i\in I$. If $I=I+g$, then $N_I$ can be identified with $N_{g-k}$ in which $k = |I|/2$. Then Thaddeus shows:

\vspace{.3cm}
\[
    \pm \prod_{j\in I} \psi_{j} \; = \; \text{P.D.} [N_I] \, \in \, H^{6g-6-3|I|}(N_g;\Z).
\]
\vspace{.3cm}

\noindent This immediately establishes (\ref{eq:psi}), up to signs. We now turn back to the moduli space $M_g$, which has the same description as does $N_g$ but with $U(2)$ replacing $SU(2)$. For $I,K\subset\{1,\ldots,2g\}$, embedded in $N_g\times J_g$ is the submanifold $N_I\times J_K$, where $N_I$ is as before, and $J_K$ consists of $2g$-tuples $(z_1,\ldots,z_{2g})$ in the $2g$-fold product of $U(1)$ such that $z_k=1$ if $k\in K$. We write $M_{IK}$ for the submanifold of $M_g$ given by the projection of $N_I\times J_K$ under the covering map $p$. It is clear that the homology class of $N_I\times J_K$ inside $N_g\times J_g$ is Poincar\'{e} dual to $\pm \prod_{i\in I}\psi_i\otimes \prod_{k\in K}\theta_k$. We compute:\\

\vspace{.4cm}

\begin{lemma} The class $b_2^i-a_1b_1^i/2$ is integral, and thus so too is $a_1b_1^i/2$. More specifically:

\vspace{.2cm}

\begin{equation}
    \pm {\text{\emph{P.D.}}}\left(\prod_{i\in I} (b_2^i-a_1b_1^i/2) \prod_{k\in K} b_1^k\right) \; = \;  [M_{IK}] \; \in \; H_{8g-6-3|I|-|K|}(M_g;\Z).\label{eq:pdm}
\end{equation}
\end{lemma}

\vspace{.45cm}

\begin{proof}
    As the statement suggests, we will ignore signs throughout. Set
    $x=\prod_{i\in I} (b_2^j-a_1b_1^j/2) \prod_{k\in K} b_1^k$, so that the above discussion implies $\text{P.D.}\left( p^\ast(x)\right) = 2^{|K|}\cdot [N_I\times J_K]$. Recall that the Poincar\'{e} dual of a cohomology class is equal to the cap product with the fundamental homology class. Also recall that the cap product satisfies the functoriality property $x\cap p_\ast(y)  = p_\ast(p^\ast(x)\cap y)$ for a homology class $y$ and a cohomology class $x$. Then we compute that $\text{P.D.}(x)$ is equal to
    
    \vspace{.2cm}
    \[
        x \, \cap \, [M_g] \; = \; x \, \cap \, 2^{-2g}p_\ast[N_g\times J_g]\; = \; 2^{-2g}p_\ast\left(p^\ast\left(x\right) \, \cap \, [N_g\times J_g]\right) \; = \; 2^{|K|-2g}p_\ast [N_I\times J_K].
    \]
    \vspace{.2cm}
    
    \noindent The final expression obtained on the right hand side is equal to $[M_{IK}]$, because $N_I\times J_K$ is clearly a $2^{2g-|K|}$ sheeted covering of $M_{IK}$.
\end{proof}

\vspace{.65cm}

\vspace{.4cm}

\noindent Note the submanifold $M_{IK}$ may be described as the subspace of $M_g$ consisting of conjugacy classes of tuples $(A_i)_{i=1}^{2g}$ of matrices in $U(2)$ whose product of commutators is $-1$, and such that $A_i\in SU(2)$ if $i\in K$ while $A_i$ is in the center of $U(2)$ if $i\in I$. We next establish:\\

\vspace{.4cm}

\begin{lemma}\label{lemma:a1rec}
If $a_1^m\in H^\ast(M_{g-1})$ is divisible by $d\in \Z$, so too is $(b_2^g-a_1b_1^g/2) b_1^{2g}a_1^m\in H^\ast(M_g)$.
\end{lemma}

\vspace{.35cm}

\begin{proof}
    From above, we know that $(b_2^g-a_1b_1^g/2) b_1^{2g}$ is Poincar\'{e} dual to $[M_{IK}]$ where $I=\{g\}$ and $K=\{2g\}$. In this subspace, $[A_g,A_{2g}]=1$, so that always $\prod_{i=1}^{g-1} [A_i,A_{i+g}]=-1$. Thus there is a well-defined map from $M_{IK}$ to $M_{g-1}$ which forgets $A_g$ and $A_{2g}$. This is a fibration with fiber $SU(2)\times S^1$, where $S^1$ is identified with the center of $U(2)$. Because conjugation does not interact with the center of $U(2)$, we may write $M_{IK} = P\times S^1$ where $P$ is an $SU(2)$-fibration over $M_{g-1}$. The fibration $P$ has a section, given by $A_{2g}=1\in SU(2)=S^3$. Thus just as in \cite{thaddeus-intro}, the euler class of $P$ vanishes, and the Gysin exact sequence for a 3-sphere fibration implies the right-hand isomorphism below:
    
    \vspace{.3cm}
    \[
        H^\ast(M_{IK};\Z) \; \cong \; H^\ast(P;\Z)\otimes H^\ast(S^1;\Z), \qquad H^\ast(P;\Z) \; \cong \; H^\ast(M_{g-1};\Z)\otimes H^\ast(S^3;\Z)
    \]
    \vspace{.1cm}
    
    \noindent While the left-hand isomorphism above is an isomorphism of graded-commutative rings, we do not know the same for the right-hand isomorphism. However, the Leray-Hirsch Theorem tells us that this latter isomorphism respects the $H^\ast(M_{g-1};\Z)$-module structures. It is a straightforward matter to verify that $a_1\in H^2(M_g;\Z)$ goes to $a_1\otimes 1\in H^2(M_{g-1};\Z)\otimes H^0(S^3;\Z)$ under this isomorphism. The lemma then follows using the $H^\ast(M_{g-1};\Z)$-module structure.
\end{proof}

\vspace{.6cm}

\begin{prop}\label{prop:2diva1}
    For $n\geqslant 2g-1$, the element $a_1^{n}\in H^{2n}(M_g;\Z)$ is divisible by $2^{n-2g+1}$.
\end{prop}

\vspace{.25cm}

\begin{proof} The proof is by induction. We assume the result holds for $a_1\in H^2(M_{k};\Z)$ for $k\leqslant g-1$. Further, we add the induction hypothesis that $a_1^q \equiv 0$ (mod $2^{q-2g+1}$) for $q>n$. Note that this is automatically true for $q$ large enough, since $a_1$ is nilpotent.\\

From Section \ref{sec:mg}, we gather that the residue classes of $a_1,d_{i},b_1^j,b_2^j$ for $2\leqslant i \leqslant 2g-1$ and $1\leqslant j\leqslant 2g$ generate the mod $2^m$ cohomology ring of $M_g$ for any $m\geqslant 1$, and in particular $m=n-2g+1$. It suffices then to show that for every partition $\lambda$, subsets $J_1,J_2\subset\{1,\ldots,2g\}$, and $\ell\geqslant 0$ we have

\vspace{.3cm}

\begin{equation}
    a_1^{n+\ell}d_{\lambda_1}\cdots d_{\lambda_k}\prod_{j\in J_1}b_1^j\prod_{j\in J_2}b_2^j [M_g] \; \equiv \; 0 \mod 2^{n-2g+1}.\label{eq:a1div}
\end{equation}

\vspace{.3cm}

\noindent The case in which $J_2$ is empty follows the argument of Proposition \ref{prop:2divalpha}, but this time using  Proposition \ref{prop:maincomputation}. We do not use any induction hypothesis here.\\

By Proposition \ref{prop:bterms}, if $J_2$ is not empty, then either the left side of (\ref{eq:a1div}) vanishes, or at least one of two kinds of terms appears: $b_2^jb_2^{j+g}$ or $b_1^jb_2^{j+g}$. Without loss of generality we will suppose $j=g$.\\

First suppose $b_1^gb_2^{2g}$ appears in (\ref{eq:a1div}). Let $x$ denote the monomial in (\ref{eq:a1div}) omitting this term and $a_1^n$. We must show $2^{n-2g+1}$ divides $a_1^nb_1^g b_2^{2g}x[M_g]$. First note that we can replace $b_2^{2g}$ by $b_2^{2g}-a_1b_1^{2g}/2$. Indeed,  $a_1^{n+1}b_1^gb_1^{2g}x[M_g]/2$ is divisible by $2^{(n+1)-2g+1}/2$ using the induction hypothesis on $n$. Finally,

\vspace{.3cm}

\[
    a_1^n(b_2^{2g}-a_1b_1^{2g}/2)x[M_g] \; \equiv \; 0 \; \mod 2^{n-2g+1}
\]

\vspace{.3cm}

\noindent using Lemma \ref{lemma:a1rec} and the induction hypothesis on $g$. Thus the case in which $b_1^gb_2^{2g}$ appears is done.\\ 

Now suppose that $b_2^g b_2^{2g}$ appears in (\ref{eq:a1div}). Just as was done in the previous case, we may replace each $b_2^j$ here with $b_2^j-a_1b_1^j/2$, and upon pulling back via the covering $p$, we get $\psi_g\psi_{2g}$, and the result follows from induction on $g$, using (\ref{eq:cup}) and Thaddeus's genus recursive formula (\ref{eq:psi}). This exhausts all cases and completes the proof of the proposition, as well as the proof of Theorem \ref{thm:nilp}.
\end{proof}

\vspace{.65cm}

\section{Computations}\label{sec:comps}

In this section we give some examples of the pairings on $N_g$ calculated by Theorem \ref{thm:maincomputation} and describe the ring structure $H^\ast(N_g;\Z_2)$ for low values of $g$. Much of this discussion can be carried out for the moduli space $M_g$, but we will not pursue this.\\

First, recall that Theorem \ref{thm:maincomputation} computes the pairings involving the classes $\xi_{g,i}$. The pairings are encoded in the Chern number polynomial $\cn(Z_g|_{N_k})$, which is equal to (\ref{eq:cnxi}). Formula (\ref{eq:ng}) easily computes this polynomial using a program such as \texttt{Sage}, which has symmetric function methods available. For example, we have $\cn(Z_1|_{N_1})=1$, $2\cn(Z_2|_{N_2}) = -m_{(1^3)} - 2m_{(21)}$, and 

\vspace{.05cm}

\begin{align*}
2^2\cn(Z_3|_{N_3}) & \; = \;14m_{(1^6)} + 17m_{(21^4)} + 26m_{(2^21^2)} + 28m_{(2^3)} + 9m_{(31^3)}\\
 & \qquad + 12m_{(321)} + 6m_{(3^2)} + 6m_{(41^2)} + 3m_{(42)}.
\end{align*}

\vspace{.35cm}

\noindent These polynomials quickly become quite lengthy. For example, if we compute the genus $4$ polynomial in terms of elementary symmetric functions $e_\lambda$ we find

\begin{align*}
    2^3\cn(Z_4|_{N_4}) \; & = \; -4e_{(2^31^3)} + 18e_{(32^21^2)} - 44e_{(3^221)} + 65e_{(3^3)} + 36e_{(421^3)} -100e_{(431^2)}\\
    &\qquad - 44e_{(521^2)} + 150e_{(531)} -20e_{(61^3)} +27e_{(71^2)}.
\end{align*}

\vspace{.35cm}

\noindent If we instead write this same polynomial in terms of monomial symmetric functions $m_\lambda$ then it has 26 non-zero terms. Similarly, the corresponding genus 5 polynomial has 20 non-zero terms when written using the $e_\lambda$, and 70 non-zero terms when using the $m_\lambda$.\\

Of course, $\cn(Z_g|_{N_k})$ is of intermediary interest to us: our goal was to compute $\cn(f_! V_g |_{N_k})$, the polynomial encoding the pairings involving the $\delta_{g,i}$ classes. We can compute these using the Chern number polynomial for $Z_g|_{N_k}$ via the transformations (\ref{eq:defxi}). These are typically more complicated, however. For example, we have $\cn(f_! V_1 |_{N_1})=1$, $\cn(f_! V_2 |_{N_2})=4m_{(1^3)} + 3m_{(21)} + m_{(3)}$, and

\vspace{.05cm}

\begin{align*}
\cn(f_! V_3 |_{N_3}) & \; = \; 14336m_{(1^6)} + 6464m_{(21^4)} + 2936m_{(2^21^2)} + 1339m_{(2^3)} + 1568m_{(31^3)}\\ & \qquad  + 722m_{(321)} + 182m_{(3^2)} + 212m_{(41^2)} + 98m_{(42)} + 14m_{(51)}
\end{align*}

\vspace{.35cm}

\noindent For genus 4, there are 28 non-zero coefficients whether we use the basis $e_\lambda$ or $m_\lambda$, while for genus 5, there are 73 non-zero coefficients in either basis. In each case, respectively, $28$ and $73$ is the number of monomials in the $\delta_i$ classes of top degree, so every possible pairing is non-zero. We have focused on the cases $g=k$ for simplicity; when $k<g$ the computations are somewhat similar.\\

When we consider the pairings only modulo 2 which are relevant for $H^\ast(N_g;\Z_2)$ the situation is considerably more manageable. First, we recall from Corollary \ref{cor:nggens} that the residue classes of $\alpha, \delta_{2^i}, \psi_j$ generate the ring $H^\ast(N_g;\Z_2)$, where $2\leqslant 2^i\leqslant 2g-1$ and $1\leqslant j \leqslant 2g$. We can obtain pairing formulas for these classes from the above data as follows. Let $\mathscr{P}(2)$ be the set of partitions each of whose parts is a power of 2. Thus $(4,2,2,1)\in \mathscr{P}(2)$ but $(6,4,2,1)\notin\mathscr{P}(2)$. For $\lambda\in \mathscr{P}(2)$ let $m_1$ denote the number of 1's in $\lambda$, and let $\lambda^\#$ denote the partition obtained from $\lambda$ by removing all of its 1's, so in particular $m_1 = |\lambda|-|\lambda^\#|$. Set $\delta_{g,\lambda} \; := \; \delta_{g,\lambda_1}\cdots \delta_{g,\lambda_n}$. Now we define the following:

\vspace{.3cm}

\[
    P_{g,k} \; := \; \sum_{\lambda\in \mathscr{P}(2)} \alpha^{m_1} \delta_{g,\lambda^\#}[N_k] \cdot m_\lambda \qquad \mod 2
\]

\vspace{.3cm} 

\noindent Then the collection of $P_{g,k}$ with $1\leqslant k \leqslant g$ determines the ring structure of $H^\ast(N_g;\Z_2)$. Indeed, it is evident that $P_{g,g}$ encodes all mod 2 pairings involving the generators $\delta_{g,2^i}$ and $\alpha$, while, for example, the pairing $\alpha^{m_1}\delta_{g,\lambda^\#}\psi_1\psi_{1+g}\cdots\psi_{g-k}\psi_{2g-k}[N_g]$ is equal to the coefficient of $m_\lambda$ in $P_{g,k}$. Recalling that $\delta_{g,1} = (g-1)\alpha$, we have the following, which tells us how to compute $P_{g,k}$ from the $\delta_{g,i}$ pairings:

\vspace{.3cm} 

\[
    \underset{m_\lambda}{\text{Coeff}}\, \Big[ P_{g,k} \Big] \; \equiv \; \underset{m_\lambda}{\text{Coeff}} \, \Big[ \cn(f_!V_g|_{N_k}) / (g-1)^{m_1} \Big] \mod 2
\]

\vspace{.3cm}

\noindent Here $\lambda\in \mathscr{P}(2)$. The polynomials $P_{g,k}$ are presented up to genus $8$ in Table \ref{table:mod2}.\\

We remark that the computations of Chern numbers and hence that of Table \ref{table:mod2} could have also been done without using Theorem \ref{thm:maincomputation}. Indeed, one can write out the $\delta_{g,i}$ classes as rational functions of $\alpha,\beta,\gamma$ using (\ref{eq:zagchern}) and (\ref{eq:defxi}), and then apply Thaddeus's intersection pairing formula for $\alpha^i\beta^j\gamma^k$ from Section \ref{sec:intersection} term-wise. As an illustration of this, we may write $\delta_8=\delta_{6,8}\in H^{16}(N_6;\Z)$ as follows:

\vspace{.3cm} 
\begin{align*}
\delta_{6,8}  \; = \; & \textstyle \left(\frac{3184129}{10321920}\right)\alpha^8 - \left(\frac{351163}{368640}\right)\alpha^6\beta + \left(\frac{747229}{737280}\right)\alpha^4\beta^2 + \left(\frac{3539}{23040}\right)\alpha^5\gamma 
 -\left(\frac{1044149}{2580480}\right)\alpha^2\beta^3\\
 & \\
& \textstyle - \left(\frac{1061}{3840}\right)\alpha^3\beta\gamma + \left(\frac{1155}{32768}\right)\beta^4 + \left(\frac{18829}{161280}\right)\alpha\beta^2\gamma
 + \left(\frac{13}{576}\right)\alpha^2\gamma^2 - \left(\frac{31}{2880}\right)\beta\gamma^2
\end{align*}
\vspace{.3cm} 

\noindent We also compute $\delta_{6,2} = \frac{91}{8}\alpha^2 - \frac{11}{8}\beta$. Then we can apply Thaddeus's intersection pairing formula to the terms of $\alpha^5\delta_{6,2}\delta_{6,8}$ and sum to obtain $117071517415$. This number is odd, and accounts for the partition $(8,2,1,1,1,1,1)$ appearing the first column of row $g=6$ in Table \ref{table:mod2}.\\

From Table \ref{table:mod2} we can read off the ring structure of $H^\ast(N_g;\Z_2)$ for $1\leqslant g\leqslant 8$, and we will spell this out for $1\leqslant g\leqslant 4$. We make a few preliminary remarks. We know from Cor. \ref{cor:nggens} that $H^\ast(N_g;\Z_2)$ is generated by $\alpha,\delta_{2^i},\psi_j$ for $2\leqslant 2^i\leqslant 2g-1$ and $1\leqslant j\leqslant 2g$. We write $I(N_g;R) \subset H^\ast(N_g;R)$ for the subring invariant under the $\text{Sp}(2g,\Z)$-action, where $R$ is any ring. It is well known that $I(N_g;\mathbb{Q})$ is generated by $\alpha,\beta,\gamma$ and that a monomial basis for the vector space $I(N_g;\mathbb{Q})$ is given by

\vspace{.2cm}
\[
    \{ \alpha^i\beta^j\gamma^k: \;\; i,j,k\geqslant 0,\;\; i+j+k < g\},
\]
\vspace{.2cm}

\noindent see for example \cite[\S 5]{siebert-tian}. In particular, $\dim I(N_g;\Q) = g(g+1)(g+2)/6 = T_g$, the $g^\text{th}$ Tetrahedral number. This in fact holds for any field, and in particular $\Z_2$. From Prop. \ref{prop:inv} we know that $I(N_g;\Z_2)$ is generated by $\alpha,\delta_{2^i},\upsilon_{2^j}$ for $2\leqslant 2^i\leqslant 2g-1$ and $1\leqslant 2^j<g$. We now proceed to describe the rings $H^\ast(N_g;\Z_2)$ and their invariant subrings for $1\leqslant g\leqslant 4$.\\

\noindent \textbf{Genus 1:} In this case, $N_1$ is a point, so $H^\ast(N_1;\Z)\cong\Z$.\\

\vspace{.3cm}

\noindent \textbf{Genus 2:} Even over $\Z$ this ring is simple to describe, cf. \cite[\S 10]{newstead-top}. As remarked there, the only interesting cup product in $H^\ast(N_2;\Z)$ is $\alpha^2$, which is $4$ times an integral generator, equal to $\alpha^2-\delta_2$. Thus the ring $H^\ast(N_2;\Z_2)$, which has betti numbers $1,0,1,4,1,0,1$, has the residue classes of $\alpha$ in degree 2 and $\delta_2$ in degree 4, and $\alpha^2=0$ (mod 2). The classes $\psi_1,\psi_2,\psi_3,\psi_4$ generate the 4-dimensional middle cohomology group, and $\alpha\delta_2$, $\psi_1\psi_3$, $\psi_2\psi_4$ are all equal to the non-zero top degree element in $H^6(N_2;\Z_2)$, while all other pairings are zero. The invariant subring $I(N_2;\Z_2)$ is generated by $\alpha$ and $\delta_2$ and has betti numbers $1,0,1,0,1,0,1$.\\

\vspace{.3cm}

\noindent \textbf{Genus 3:} The ring $H^\ast(N_3;\Z_2)$ has betti numbers $1,0,1,6,2,6,16,6,2,6,1,0,1$. It is generated by $\alpha,\delta_2,\delta_4$ and $\psi_j$ for $1\leqslant j \leqslant 6$. The nontrivial pairings in top degree, as can be read from Table \ref{table:mod2}, are

\vspace{.2cm}
\[
    \alpha^2\delta_4, \quad \delta_2^3, \qquad \psi_j\psi_{j+g}\alpha\delta_2\; (1\leqslant j\leqslant 3), \qquad \psi_j\psi_{j+g}\psi_k\psi_{k+g} \; ( 1\leqslant j\neq k \leqslant 3)
\]
\vspace{.2cm}

\noindent The invariant subring $I(N_3;\Z_2)$, which has betti numbers $1,0,1,0,2,0,2,0,2,0,1,0,1$, is generated by $\alpha,\delta_2,\delta_4$ and $\upsilon_1,\upsilon_2$. We can compute a presentation for the invariant ring:

\vspace{.2cm}
\[
I(N_3;\Z_2) \; \cong \;
\Z_2[\alpha,\;\delta_2,\;\delta_4,\;\upsilon_1,\upsilon_2]/(\upsilon_1^2, \; \delta_4\upsilon_1,\; \delta_4^2,\; \delta_2\delta_4, \; \delta_1\delta_4+\delta_2\upsilon_1,\; \delta_2^2+\delta_1\upsilon_1,\; \delta_1^2\upsilon_1,\; \delta_1^2\delta_2, \;\delta_1^3)
\]
\vspace{.2cm}

\noindent We remind the reader that $\upsilon_1 = \psi_1\psi_4 + \psi_2\psi_5 + \psi_3\psi_6$ and $\upsilon_2 = \psi_1\psi_4\psi_2\psi_5 + \psi_1\psi_4\psi_3\psi_6 + \psi_2\psi_5\psi_3\psi_6$. Note here that $\delta_2^3$ is nonzero. This property seems to possibly persist for all $\delta_{g,2}\in H^\ast(N_g;\Z_2)$, and can perhaps be proven using the same methods used to prove Theorem \ref{thm:nilp}.\\

\vspace{.3cm}

\noindent \textbf{Genus 4:} The ring $H^\ast(N_3;\Z_2)$ has betti numbers $1,0,1,8,2,8,30,16,30,64,30,16,30,8,2,8,1,0,1$. It is generated by $\alpha,\delta_2,\delta_4$ and $\psi_j$ for $1\leqslant j\leqslant 8$. The only non-trivial pairings in the top degree, as read from Table \ref{table:mod2}, are the following, in which $1\leqslant j,k,\ell\leqslant 4$ are distinct:

\vspace{.2cm}
\[
    \alpha\delta_2^2\delta_4, \quad \alpha^3\delta_2\delta_4, \quad \psi_j\psi_{j+g}\alpha^2\delta_4, \quad  \psi_j\psi_{j+g}\delta_2\delta_4,
\]
\[
    \psi_j\psi_{j+g}\delta_2^3, \quad  \psi_j\psi_{j+g}\psi_k\psi_{k+g}\alpha\delta_2, \quad \psi_j\psi_{j+g}\psi_k\psi_{k+g}\psi_\ell\psi_{\ell+g}
\]
\vspace{.2cm}

\noindent The invariant ring $I(N_4;\Z_2)$ has betti numbers $1,0,1,0,2,0,3,0,3,0,3,0,3,0,2,0,1,0,1$. It is generated by $\alpha,\delta_2,\delta_4,\upsilon_1,\upsilon_2$, just like the genus 3 case. The ideal of relations here is generated by:

\vspace{.2cm}
\[
\alpha^4, \; \alpha^2 \delta_2 + \alpha \upsilon_1 + \delta_2^2,\; \delta_2 \alpha^3 + \alpha \delta_2^2, \alpha^2 \upsilon_1,\; \alpha^2 \delta_2^2,\; \alpha \upsilon_2,
\]
\[ \upsilon_1^2,\; \upsilon_2^2,\; \delta_2^3 \upsilon_1+\alpha^2 \delta_4 \upsilon_1,\; \delta_2^3 + \upsilon_2,\; \upsilon_1 \upsilon_2,\; \delta_4 \upsilon_1,\; \delta_4^2
\]
\vspace{.2cm}

\noindent We will stop here, but the interested reader can proceed to describe the higher genus cases up to $g=8$ using Table \ref{table:mod2}. Also, one can similarly describe the rings $H^\ast(N_g;\Z_p)$ for other primes $p$ using our Chern number computations and some additional work. \\

\begin {table}
\begin{center}
\caption {Partitions $\lambda\in \mathscr{P}(2)$ for which $\underset{m_\lambda}{\text{Coeff}}\,[ P_{g,k}]\equiv 1$ (mod 2)}\label{table:mod2}
\vspace{.3cm}
\scalebox{0.9}{\begin{tabular}{ ccccccccc } 
    & $k=0$ & $k=1$ & $k=2$ & $k=3$ & $k=4$ & $k=5$ & $k=6$ & $k=7$ \\
\toprule
$g=1$  & $0$ &   &   &   &   &   &   &   \\
&$\phantom{(111111)}$&$\phantom{(111111)}$&$\phantom{(111111)}$&$\phantom{(111111)}$&$\phantom{(111111)}$&$\phantom{(111111)}$&$\phantom{(111111)}$&$\phantom{(111111)}$\\
$g=2$  & $(2^1 1^1)$ & $0$  &   &   &   &   &   &   \\
&&&&&&&&\\
$g=3$  & $(4^1 1^2)$ & $(2^1 1^1)$ & $0$ &  &  &  &  &  \\
       & $(2^{3})$ & & & & & & &\\
       &&&&&&&&\\
$g=4$  & $(4^1 2^2 1^1)$ & $(4^1 1^2)$ & $(2^1 1^1)$ & $0$ &  &  &  &  \\
       & $(4^1 2^1 1^3)$  & $(4^1 2^1)$ &  & & & & & \\
       & &  $(2^{3})$  & & & & & & \\
       &&&&&&&&\\
$g=5$  & $(8^1 2^1 1^2)$  & $(4^1 2^1 1^3)$ &  $(4^1 1^2)$ & $(2^11^1)$ & $0$ &  &  &  \\
       & $(4^1 2^3 1^1)$ & $(4^1 2^2 1^1)$ & $(2^1 4^1)$ & & & & & \\
       & $(8^1 1^4)$ & & $(2^3)$ & & & & & \\
       & $(4^3)$ &&&&&&&\\
&&&&&&&&\\
$g=6$  & $(8^1 2^2 1^3)$ & $(4^1 2^3 1^2)$ & $(4^2 2^1 1^3)$ & $(4^1 1^2)$ & $(2^1 1^1)$ & $0$ &  &  \\
       & $(8^1 2^1 1^5)$ & $(8^1 2^2)$  & $(4^1 2^2 1^1)$  & $(2^3)$ & & & & \\
       & $(8^1 2^3 1^1)$ & $(8^1 1^4)$ & & & & & & \\
       & $(4^3 2^1 1^1)$ & $(4^3)$ & & & & & & \\
&&&&&&&&\\
$g=7$  & $(8^1 4^1 2^2 1^2)$ &  $(8^1 4^1 2^1 1^1)$  & $(4^1 2^3 1^2)$ & $(4^1 2^1 1^3)$ & $(4^1 1^2)$ & $(2^1 1^1)$  &  $0$ &  \\
       & $(8^1 2^3 1^4)$ &$(8^1 2^2 1^3)$& $(8^1 1^4)$ & $(4^1 2^2 1^1)$ & $(2^3)$ & & &\\
       & $(8^1 4^2 1^2)$ & $(8^1 2^3 1^1)$ & $(4^3)$ & & & & &\\
       & $(8^1 4^1 1^6)$ & $(8^1 4^1 1^3)$ & & & & & &\\
       & $(4^3 2^3)$ & $(4^3 2^1 1^1)$ & & & & & &\\
       & $(8^1 2^5)$ & $(8^1 2^1 1^5)$ & & & & & &\\
&&&&&&&&\\
$g=8$  & $(8^1 4^1 2^3 1^3)$ & $(8^1 4^1 2^1 1^4)$  & $(8^1 4^1 2^1 1^1)$ & $(8^1 2^1 1^2)$  & $(4^1 2^2 1^1)$ & $(4^1 2^1)$ & $(2^1 1^1)$ & $0$ \\
      & $(8^1 4^1 2^2 1^5)$ & $(8^1 4^2 2^1)$ & $(8^1 4^1 1^3)$ & $(4^1 2^3 1^2)$  &  $(4^1 2^1 1^3)$ & $(4^1 1^2)$ & &\\
      & $(8^1 4^2 2^11^3)$ & $(8^1 4^2 1^2)$ & $(8^1 2^3 1^1)$ & $(8^1 4^1)$ & & $(2^3)$ & &\\
      & $(8^1 4^2 2^2 1^1)$ & $(8^1 4^1 2^3)$ & $(8^1 2^2 1^3)$ & $(8^1 2^2)$ & & & &\\
      & $(8^1 4^1 2^4 1^1)$ & $(8^1 4^1 1^6)$ & $(8^1 2^1 1^5)$ & $(8^1 1^4)$ & & & &\\
      & $(8^1 4^1 2^1 1^7)$ & $(8^1 2^3 1^4)$ & $(4^3 2^1 1^1)$ & $(4^3)$ & & & &\\
      && $(8^1 2^5)$ &&&&&&\\
      && $(4^3 2^3)$ &&&&&&\\
\bottomrule
\end{tabular}}
\end{center}
\caption*{\small{A partition $(8^a 4^b 2^c 1^d)$ appears in row $g$ and column $k$ of this table if and only if the monomial $\mu = \phi\delta_8^a\delta_4^b\delta_2^c\alpha^d$ is nonzero in the ring $H^\ast(N_g;\Z_2)$, i.e. $\mu[N_g]\equiv 1$ (mod 2), where $\phi = \psi_1\psi_{1+g}\cdots \psi_{k}\psi_{k+g}$.}}
\vspace{.2cm}
\end{table}

\vspace{.65cm}

\appendix
\section{Background on symmetric functions}\label{appendix}

In this section we provide the reader with the relevant background material on symmetric polynomials. For details and proofs, see \cite{mac}. We will typically work with symmetric functions in infinitely many variables $x_1,x_2,x_3,\ldots$ with either integer or rational coefficients.\\

There are a few standard symmetric functions that will be of use to us. First, for any positive integer $n$, we have the {\emph{elementary}} symmetric function $e_n$, given by

\vspace{.18cm}
\[
    e_n \; = \; \sum_{i_1 < i_2 <\cdots < i_n} x_{i_1}x_{i_2}\cdots x_{i_n}.
\]
\vspace{.1cm}

\noindent If $\lambda = (\lambda_1,\lambda_2,\ldots,\lambda_k)$ is a partition, i.e. a nonincreasing sequence of nonnegative integers, then we define $e_{\lambda} = e_{\lambda_1}e_{\lambda_2}\cdots e_{\lambda_k}$. If in the definition of $e_n$ one sums over $i_1\geqslant i_2 \geqslant \cdots \geqslant i_n$ instead, the result is the {\emph{complete}} symmetric function $h_n$, and we may similarly define $h_{\lambda}$. For $n=0$, set $e_0=h_0=1$.\\

Next, for any given partition $\lambda$, we have the {\emph{monomial}} symmetric function $m_\lambda$, which is the sum of all distinct monomials of the form $x_{i_1}^{\lambda_1}x_{i_2}^{\lambda_2}\cdots x_{i_k}^{\lambda_k}$ in which $i_1,\ldots,i_k$ are distinct. Although we do not make much use of them, we also define the {\emph{power sum}} symmetric function $p_n$ by

\vspace{.18cm}
\[
    p_n \; = \; \sum_{i\geqslant 0} x_i^{n}.
\]
\vspace{.1cm}

\noindent In Section \ref{sec:skew} we define (skew) Schur symmetric functions $s_\lambda$. It is often convenient to write a partition $\lambda=(\lambda_1,\lambda_2,\ldots,\lambda_k)$ in the alternative format $\lambda = (1^{m_1}2^{m_2}\cdots k^{m_k})$ in which $\lambda$ has $m_i$ number of parts equal to $i$. For example, the partition $(2,2,1,1,1)$ can be written instead as $(1^{3}2^2)$. We write $|\lambda| = \sum_{i=1}^k \lambda_k$ for the sum of a partition, and $l(\lambda)=k$ for its length. Sometimes we insert commas for clarity; the last partition may be written as $( 2^2 \, 1^3)$.\\

We write $\Lambda$ for the ring of symmetric functions with integer coefficients. The Fundamental Theorem of Symmetric Functions says that $\Lambda$ is isomorphic to the ring $\Z[e_1,e_2,\ldots]$ freely generated by the $e_i$. The statement also holds with the $e_i$ replaced by $h_i$. Also, the sets 

\vspace{.18cm}
\[
    \{e_\lambda\}, \quad \{h_\lambda\}, \quad \{m_\lambda\}, \quad \{s_\lambda\},
\]
\vspace{.1cm}

\noindent where $\lambda$ runs over all partitions, each separately provides an additive basis for $\Lambda$. If we work instead with rational coefficients, then the ring of symmetric functions $\Lambda\otimes_\Z \Q$ is isomorphic to the freely generated algebra $\Q[p_1,p_2,\ldots]$, and $\{p_\lambda\}$ provides an additive basis for the vector space $\Lambda\otimes_\Z\Q$.

\vspace{.65cm}

\bibliographystyle{alpha}
\bibliography{main}

\end{document}